\ifdefined\PreambleLoaded\else
  \documentclass{amsart}
\usepackage{marktext}
\usepackage{gimac}

\usepackage{mdframed}
\usepackage{algpseudocode, algorithm, amsmath}

\newmdenv[
  linewidth=0.5pt,
  skipabove=10pt,
  skipbelow=10pt,
  backgroundcolor=gray!3,
  linecolor=gray!60,
  innertopmargin=8pt,
  innerbottommargin=8pt,
  roundcorner=4pt,
  nobreak=false  
]{algbox}

\newcommand{\GI}[2][]{\sidenote[colback=yellow!20]{\textbf{GI\xspace #1:} #2}}
\newcommand{\SJ}[2][]{\sidenote[colback=green!10]{\textbf{SJ\xspace #1:} #2}}

\newcommand{\Bias}{\mathrm{\tilde{B}ias}}
\newcommand{\Err}{\mathrm{Err}}

\newcommand{\mix}{\mathrm{mix}}

\DeclareMathOperator{\dist}{dist}

\DeclareMathOperator{\Hess}{Hess}

\DeclarePairedDelimiter{\ip}{\langle}{\rangle}
\newcommand{\osc}{\mathrm{osc}}
\DeclareMathOperator{\ESS}{ESS}

\newtheorem{assumption}[theorem]{Assumption}
\newtheorem{property}[theorem]{Property}

\newcommand{\PreambleLoaded}{}

\fi

\begin{document}

\title[Convergence of Langevin AIS]
{Convergence of Langevin AIS for multimodal distributions}

\author[Agarwal]{Akshat Agarwal}
\address{%
  Princeton University, Princeton NJ 08544.
}
\email{akshat.agarwal@princeton.edu}

\author[Iyer]{Gautam Iyer}
\address{%
  Department of Mathematical Sciences, Carnegie Mellon University, Pittsburgh, PA 15213.
}
\email{gautam@math.cmu.edu}

\author[Jameson]{Aidan Jameson}
\address{%
  University of Utah, Salt Lake City, UT 84112
}
\email{u0680511@utah.edu}

\author[Son]{Seungjae Son}
\address{%
  Department of Mathematical Sciences, Carnegie Mellon University, Pittsburgh, PA 15213.
}
\email{seungjas@andrew.cmu.edu}

\author[Wimmer]{Wyatt Wimmer}
\address{%
  Brigham Young University, Provo, UT 84602.
}
\email{wyattwimmer@gmail.com}

\begin{abstract}
  We study convergence rates of the \emph{annealed importance sampling algorithm} (Neal '01) combined with \emph{Langevin Monte Carlo} when the target is a multimodal Gibbs measure.
  The main result shows that for a fixed error threshold, the time complexity is \emph{quadratic in the inverse temperature}.
  We identify a simple and useful quantity that controls the sampling error for AIS in a general setting, and then bound this quantity in our setting using spectral estimates.
  We also study an autonormalized version and obtain bounds for the time complexity in terms of the inverse temperature.
\end{abstract}

\thanks{This work has been partially supported by the National Science Foundation under grants
  DMS-2406853,
  DMS-2342349
  and the Center for Nonlinear Analysis.}
\maketitle

\section{Introduction}


\subsection{Main Results}
We begin by stating our main results, following which we survey the literature and place our results in the context of the current literature.
Let~$\mathcal X$ be a configuration space and~$U \colon \mathcal X \to \R$ be an energy function.
Given~$\epsilon > 0$, the Gibbs measure with temperature~$\epsilon$ is defined by
\begin{equation}\label{e:piDef}
  \pi_\epsilon \defeq \frac{\tilde \pi_\epsilon}{Z_\epsilon}
  \quad\text{where}\quad
  \tilde \pi_\epsilon = e^{-U / \epsilon}
  \,,
  \quad\text{and}\quad
  Z_\epsilon = \int_{\mathcal X} \tilde \pi_\epsilon \, dx
  \,,
\end{equation}
where the last integral is carried out with respect to some reference measure on~$\mathcal X$.
We typically only consider the case where~$\mathcal X$ is the~$d$-dimensional Euclidean space~$\R^d$, or the~$d$-dimensional torus~$\T^d$.
In these cases the reference measure in~\eqref{e:piDef} is the Lebesgue measure.

Consider the overdamped Langevin equation
\begin{equation}\label{e:langevin}
  d X^\epsilon_t = -\grad U(X^\epsilon_t) + \sqrt{2\epsilon} \, dB_t
  \,,
\end{equation}
where~$B$ is a standard~$d$-dimensional Brownian motion.
It is well known (see for instance~\cite{Pavliotis14}) that the stationary distribution of the Langevin equation is~$\pi_\epsilon$, and the \emph{Langevin Monte Carlo (LMC)} algorithm obtains samples from~$\pi_\epsilon$ by simulating~\eqref{e:langevin} for long time.

The disadvantage of LMC is that when~$U$ is not convex, the rate at which solutions to~\eqref{e:langevin} converge to the stationary distribution is exponentially small in~$1/\epsilon$.
This is the Arrhenius law~\cite{Arrhenius89}, and is described in more detail below.
We show that this slow rate of convergence can be overcome if one uses \emph{annealed importance sampling (AIS)~\cite{Neal01}}, or an autonormalized version of AIS.

\subsubsection{Langevin Annealed Importance Sampling}

For convenience, we first state the Langevin AIS algorithm.
\begin{algorithm}[H]
  \caption{Langevin Annealed Importance Sampling (Langevin AIS)}\label{a:aisLMC}
  \begin{algorithmic}[1]
    \item[\textbf{Tunable parameters:}]
    \Statex
      \begin{enumerate}[(1)]
	\item Target temperature~$\epsilon > 0$.
	\item Annealing schedule $K \in \N$ and~$\epsilon_1 > \cdots > \epsilon_{K+1} = \epsilon$
	\item LMC simulation time~$T$.
      \end{enumerate}
    \State Start with~$X_{0} \in \T^d$ arbitrary, and weight~$\tilde w_0 = 1$.
    \For{$k=1, \dots, K$}
      \State Simulate~\eqref{e:langevin} with temperature~$\epsilon_k$, starting from~$X_{k-1}$, and let~$X_k$ be the state after time~$T$.
      \State $\tilde w_{k} \gets \tilde w_{k-1} \frac{\tilde \pi_{\epsilon_{k+1}}(X_k)}{\tilde \pi_{\epsilon_k}(X_{k})}$.
	
    \EndFor
    \State \Return~$X_{K}$ and the \emph{unnormalized} weight~$\tilde w_K$.
  \end{algorithmic}
\end{algorithm}

Our main result shows that if we obtain~$N$ independent samples using Algorithm~\ref{a:aisLMC}, and empirically normalize the weights, then we obtain good samples from~$\pi_\epsilon$ with time complexity that is quadratic in~$1/\epsilon$.

\begin{theorem}\label{t:langevin}
  Suppose the configuration space~$\mathcal X = \T^d$, and~$U$ is a regular double-well potential with non-degenerate wells of equal depth.
  Given~$\nu > 0$, and~$\epsilon_1 > 0$ there exist constants~$C_T = C_T(\nu, U, \epsilon_1)$, $\bar C_w(\nu, U, \epsilon_1)$ such that the following holds.

  For every~$\delta > 0$, $\epsilon \in (0, \epsilon_1)$, choose
  \begin{align}
    \label{e:K-choice-emp-intro}
    K &= \ceil[\bigg]{ \frac{1}{\epsilon \nu} }
    \,,
    \\
    \label{e:T-choice-emp-intro}\noeqref{e:T-choice-emp-intro}
    T &> \max\set[\Big]{
      C_T \paren[\Big]{\frac{1}{\epsilon} + \log K},
      (4+\abs{\log_2 \delta})t_{\mix, \epsilon_1}^\infty
    }\,,
    \\
    \label{e:N-choice-emp-intro}
    N &= \ceil[\bigg]{\frac{64 \bar C_w}{\delta^2}}
    \,.
  \end{align}
  Choose~$\epsilon_2$, \dots, $\epsilon_{K+1} = \epsilon$ so that~$\set{1/\epsilon_k}_{1 \leq k \leq K+1}$ are linearly spaced.
  Let~$(X^1_K, \tilde w^1_K)$, \dots, $(X^N_K, \tilde w^N_K)$ be the points and unnormalized weights obtained from~$N$ independent runs of Algorithm~\ref{a:aisLMC} with these parameters.
  Define the empirical measure~$\mu_N$ by
  \begin{equation}\label{e:muN-def}
    \mu_N \defeq \sum_{i=1}^N w_i \delta_{X^i_K}
    \quad\text{where}\quad
    w_i = \frac{\tilde w^i_K}{\tilde W_{K, N}}
    \quad\text{and}\quad
    \tilde W_{K, N} \defeq \sum_{i=1}^N \tilde w^i_K
    \,.
  \end{equation}
  Then for every bounded test function~$f$
  \begin{equation}\label{e:emp-error-bd}
    \E \abs[\big]{
      \ip{f, \mu_N} - \ip{f, \pi_\epsilon}
    }^2
      < \norm{f}_\osc^2 \delta^2
      \,.
  \end{equation}
  Moreover, for every~$s > d/2$ there exists an explicit dimensional constant~$C_s$ (independent of~$\epsilon_1, \epsilon, \nu, \delta$) such that
  \begin{equation}\label{e:emp-HsBd}
    \E \norm{\mu_N - \pi_\epsilon}_{\dot H^{-s}}^2 \leq C_s \delta^2
    \,.
  \end{equation}
\end{theorem}

Here~$t_{\mix, \epsilon_1}^\infty$ is the \emph{uniform mixing time}  of~\eqref{e:langevin} on~$\T^d$ with~$\epsilon = \epsilon_1$ (see~\cite{LevinPeres17}, or~\eqref{e:umixDef}, below).
The notation~$\ip{f, \mu}$ used in~\eqref{e:emp-error-bd} is
\begin{equation}
  \ip{f, \mu} = \int_{\T^d} f \, d\mu
  \,,
\end{equation}
  and~$\norm{\cdot}_{\dot H^{-s}}$ denotes the norm in the homogeneous Sobolev space with index~$-s$ (for instance see~\cite{ArmstrongKuusiEA19}, or~\eqref{e:HsNormDef}, below).

\begin{remark}[Effective sample size]\label{r:essLangevinIntro}
  When working with weighted points as above, it is important to ensure that the weights don't concentrate on a few particles reducing the overall efficiency.
  This is often measured by the \emph{effective sample size} (see for instance Section~8.6 in~\cite{ChopinPapaspiliopoulos20}) defined by
  \begin{equation}\label{e:ess}
    \ESS(w_1, \dots, w_N)
      \defeq \frac{1}{\sum_{i = 1}^N w_i^2}
      \,.
  \end{equation}
  We will show (see Proposition~\ref{p:essBound} and Remark~\ref{r:essLangevin}, below) that Theorem~\ref{t:langevin} implies that the effective sample size is bounded by
  \begin{equation}
    \E \ESS(w_1, \dots, w_N)
      =
      \E\paren*{\frac{1}{\sum_{i=1}^N w_i^2}} \geq \frac{N}{8\paren*{4\bar C_w +1}}\,.
  \end{equation}
\end{remark} 

\begin{remark}[Computational complexity]\label{r:complexityAis}
  We now discuss the asymptotic behavior of the computational complexity as the temperature~$\epsilon \to 0$, and as the allowed error~$\delta \to 0$.
  Assuming the computational cost of simulating~\eqref{e:langevin} for time~$T$ is~$O(T)$, and neglecting the discretization error, the computational cost of Algorithm~\ref{a:aisLMC} to achieve the error bounds~\eqref{e:emp-error-bd} or~\eqref{e:emp-HsBd} is~$O(K T N)$.
  Unravelling the~$\epsilon$ and~$\delta$ dependence in~\eqref{e:K-choice-emp-intro}--\eqref{e:N-choice-emp-intro} this implies
  \begin{equation}\label{e:complexityLAIS}
    \operatorname{complexity}(\operatorname{Algorithm}\ref{a:aisLMC})
      = O(KTN)
      \leq \frac{C_d}{\delta^2} \paren[\Big]{
	\frac{1}{\epsilon^2} + \abs{\ln \delta} 
      }
      \,,
  \end{equation}
  for some constant~$C_d = C_d(U)$.
  For reference we mention that the computational complexity of Langevin Monte Carlo under these conditions is~$e^{O(1/\epsilon)}$, and the complexity of rejection sampling is~$O(1/\epsilon^d)$.
  A more comprehensive discussion and comparison is at the end of Section~\ref{s:litReview}, below.
\end{remark}

\begin{remark}[Multiple wells, and wells of unequal depth.]\label{r:unequal}
  The assumption that~$U$ is a double well potential with wells of equal depth can be relaxed.
  If the wells have ``nearly equal depth'' so that the distribution remains truly multimodal in a temperature range, then Theorem~\ref{t:langevin} will still hold.
  The required assumptions and generalized theorem is Theorem~\ref{t:langevinGen} in Section~\ref{s:aisConv}, below.
  If there are more than two wells, then Theorem~\ref{t:langevin} will still hold provided we make the same non-degeneracy assumptions as in Section 10 in~\cite{HanIyerEA26}.
\end{remark}

\begin{remark}
  The constants~$C_T$ and~$\bar C_w$ in Theorem~\ref{t:langevin} involve dimensional factors that arise in spectral estimates and Sobolev embedding theorems.
  As a result, their dimensional dependence is not explicit (see~\eqref{e:CTdef} and~\eqref{e:T0-tildeCw-choice}, below).
  In special cases where $U$ has a low dimensional structure, the dimensional dependence of these constants can be controlled for an idealized model problem (similar to Proposition 3.1 in~\cite{HanIyerEA26}), but their dimensional dependence is not explicit for the full Langevin system.
\end{remark}

\begin{remark}
  In practice, one often has a target distribution of the form~$\pi \propto e^{-V}$ which is hard to sample from.
  By the Arrhenius law, the complexity of using LMC to directly sample from~$\pi \propto e^{-V}$ is~$e^{O(H)}$ where~$H$ is the energy barrier.
  If instead if we choose~$\epsilon = 1/H$ and $U = \epsilon V$, and use Langevin AIS to sample from~$\pi = \pi_\epsilon \propto e^{-U / \epsilon}$, then by Remark~\ref{r:complexityAis} the time complexity is now only~$O(1/\epsilon^2) = O(H^2)$.
\end{remark}

The proof of Theorem~\ref{t:langevin} consists of two steps: The first step is a general result concerning AIS.
Given a family of distributions~$\pi_1$, \dots, $\pi_{K+1}$, with corresponding Markov transition kernels~$P_1$, \dots, $P_{K+1}$, we estimate  the variance of using AIS to sample from~$\pi_{K+1}$ in terms of
\begin{equation}\label{e:Ptrk2-ub-intro}
  C_P(T) \defeq
  \prod_{k = 1}^K
    \norm{P_{k}^T r_{k}^2}_{L^\infty}
    \,,
  \quad\text{where}\quad
  r_{k} \defeq \frac{\pi_{k+1}}{\pi_{k}}
  \,.
\end{equation}
This is the content of Proposition~\ref{p:expectedDiff} and Theorem~\ref{t:ais}, below.
We will shortly examine the quantity~$C_P(T)$ further and note its resemblance to the product of the~$\chi^2$-divergences~$\chi^2(\pi_{k+1}; \pi_k)$ in Section~\ref{s:chi2}, below.

The second step in the proof of Theorem~\ref{t:langevin} bounds the quantity~$C_P(T)$ for the specific sequence of intermediate sequence of distributions used in Theorem~\ref{t:langevin}.
We do this using certain spectral properties of the generator of~\eqref{e:langevin}.
These properties are a combination of results in~\cite{
  Kolokoltsov00,
  BovierGayrardEA05,
  MenzSchlichting14
}, and are collected in a convenient form in~\cite{HanIyerEA26}.
The full proof is presented in Section~\ref{s:langevin}, below.

\subsubsection{Autonormalized Langevin AIS}

We also consider an \emph{autonormalized} version of Algorithm~\ref{a:aisLMC}.
Namely, instead of returning the unnormalized weights~$\tilde w_K$, we run~$N$ independent realizations and normalize the weights using an empirical average.
This is a commonly used idea in many \emph{Sequential Monte Carlo} samplers (see for instance~\cite{DoucetFreitasEA01,
  Liu08,
  ChopinPapaspiliopoulos20
}), and performs well in many situations of practical interest.
For convenience, we now state the version of this algorithm precisely.

\begin{algorithm}[H]
  \caption{Autonormalized Langevin Annealed Importance Sampling}\label{a:aisAutoNormalized}
  \begin{algorithmic}[1]
    \item[\textbf{Tunable parameters:}]
    \Statex
      \begin{enumerate}[(1)]
	\item Target temperature~$\epsilon > 0$.
	\item Annealing schedule $K \in \N$ and~$\epsilon_1 > \cdots > \epsilon_{K+1} = \epsilon$
	\item LMC simulation time~$T$.
	\item Number of particles~$N \in \N$.
      \end{enumerate}

    \State Start with~$w^1_0 = \cdots w^N_0 = 1/N$, and $X^1_0$, \dots, $X^N_0$ arbitrary.
    \For{$k=1, \dots, K$}
      \For{$i = 1, \dots, N$}
	\State Simulate~\eqref{e:langevin} with temperature~$\epsilon_k$, starting from~$X^i_{k-1}$, and let~$X^i_k$ be the state after time~$T$.
	\State $\tilde w^i_{k} \gets w^i_{k-1} \frac{\tilde \pi_{\epsilon_{k+1}}(X^i_k)}{\tilde \pi_{\epsilon_k}(X^i_{k})}$.
      \EndFor
      \State Let~$\tilde W_k = \sum_{j=1}^N \tilde w^j_k$, and for each~$i \in \set{1, \dots, N}$ set~$w^i_k = \tilde w^i_k / \tilde W_k$.
    \EndFor
    \State \Return~$X^1_K$, \dots, $X^N_K$ and the normalized weights~$w^1_K$, \dots, $w^N_K$.
  \end{algorithmic}
\end{algorithm}

\begin{theorem}\label{t:langevinGen-anais-intro}
  Suppose the configuration space~$\mathcal X = \T^d$, and~$U$ is a regular double-well potential with non-degenerate wells of equal depth.
  Given~$\alpha, \nu > 0$, and~$\epsilon_1 > 0$, there exists constants~$ C_N$ and~$\hat C_T$ such that the following holds.

  For every~$\delta > 0$ and~$\epsilon \in (0, \epsilon_1)$, choose
  \begin{align}
    \label{e:K-choice-emp-anais-intro}
    K &= \ceil[\bigg]{ \frac{1}{\epsilon \nu} }
    \,,
    \\
    \label{e:T-choice-emp-anais-intro}
    T &> \hat C_T\paren[\bigg]{K^{(1+\alpha)} + \frac{1}{\epsilon} + \log\paren[\Big]{\frac{1}{\delta}} + \log N}
    \,,
    \\
    \label{e:N-choice-emp-anais-intro}
    N &= \ceil[\bigg]{\frac{ C_N}{\delta^2}}K^2
    \,.
  \end{align}
  Let~$\epsilon_2$, \dots, $\epsilon_{K+1} = \epsilon$ be such that~$\set{1/\epsilon_k}_{1 \leq k \leq K+1}$ are linearly spaced.
  Run Algorithm~\ref{a:aisAutoNormalized} with these parameters, and let~$\set{X^i_K, w^i_K}_{1 \leq i \leq N}$ be the~$N$ points and normalized weights returned.
 Define the empirical measure~$\mu_N$ by
  \begin{equation}
    \mu_{N} \defeq \sum_{i=1}^N w^i_K \delta_{X^i_K}\,.
  \end{equation}
  Then, for every bounded measurable test function~$f$, we have
  \begin{equation}
    \E\abs[\big]{
      \ip{f, \mu_{N}} - \ip{f, \pi_{\epsilon}}
    }^2
    \leq \norm{f}_\osc^2 \delta^2
    \,.
  \end{equation}
  Consequently, for every~$s > d/2$ there exists an explicit dimensional constant~$\hat C_s$ (independent of~$\epsilon_1, \epsilon, \alpha, \nu, \delta$) such that
  \begin{equation}
    \E\norm[\big]{\mu_{N} - \pi_{\epsilon}}_{H^{-s}}^2
      \leq C_s \delta^2
      \,.
  \end{equation}
\end{theorem}
\begin{remark}[Effective sample size]
  We will also estimate the effective sample size and show that there exists a constant~$C_1 = C_1(\nu, \epsilon_1)$ such that 
  \begin{equation}
    \E \ESS(w^1_{k}, \dots, w^N_{k}) \geq \frac{N}{C_1\bar C_w^2}\,,
  \end{equation}
  where~$\bar C_w$ is the constant from Theorem~\ref{t:langevin}.
  This is the content of Proposition~\ref{p:essANLangevin}, below.
\end{remark}

\begin{remark}[Computational complexity]
  We now discuss the asymptotic behavior of the computational complexity as the temperature~$\epsilon \to 0$, and as the allowed error~$\delta \to 0$.
  As with Remark~\ref{r:complexityAis}, the computational cost of Algorithm~\ref{a:aisAutoNormalized} is~$O(K T N)$.
  Unravelling the~$\epsilon$ and~$\delta$ dependence in~\eqref{e:K-choice-emp-anais-intro}--\eqref{e:N-choice-emp-anais-intro} this implies
  \begin{equation}\label{e:complexityLAAIS}
    \operatorname{complexity}(\operatorname{Algorithm}\ref{a:aisAutoNormalized})
      = O(KTN)
      \leq \frac{C_d}{\delta^2 \epsilon^3} \paren[\Big]{
	\frac{1}{\epsilon^{1 + \alpha}} + \abs{\ln \delta} 
      }
      \,.
  \end{equation}
  While this is much smaller than the complexity of LMC (which is~$e^{O(1/\epsilon)}$), or rejection (which is~$O(1/\epsilon^d)$), it is larger than the complexity of Algorithm~\ref{a:aisLMC} which is~$O(1/\epsilon^2)$.
  In practice, Langevin AIS and autonormalized Langevin AIS perform comparably, and the reason the bound~\eqref{e:complexityLAAIS} is worse than~\eqref{e:complexityLAIS} may be due to suboptimality of of our estimates.
\end{remark}

The proof of Theorem~\ref{t:langevinGen-anais-intro} is a little more involved than Theorem~\ref{t:langevin} because the processes~$\set{X^i_K, w^i_K \st 1 \leq i \leq N}$ are not independent and identically distributed, but only exchangeable.
As a result, we are presently unable to deduce Theorem~\ref{t:langevinGen-anais-intro} from a general one particle result, as we will do for Theorem~\ref{t:langevin}.
Moreover, the lack of independence introduces a few technical difficulties in the proof and this leads to bounds that are worse than those in Theorem~\ref{t:langevin}.
Explicitly, Theorem~\ref{t:langevinGen-anais-intro} the choice~\eqref{e:K-choice-emp-anais-intro}--\eqref{e:N-choice-emp-anais-intro} requires~$N$ to grow like~$1/\epsilon^2$, where as in Theorem~\ref{t:langevin} one can choose~$N$ independent of~$\epsilon$.
Moreover, the simulation time~$T$ in~\eqref{e:T-choice-emp-anais-intro} now grows like~$1/\epsilon^{1 + \alpha}$, where as in Theorem~\ref{t:langevin} it only needed to grow like~$1/\epsilon$.
In spite of the difference in the provable bounds in this situation, Algorithm~\ref{a:aisAutoNormalized} performs well in practice and is used often~\cite{DoucetFreitasEA01,
  Liu08,
  ChopinPapaspiliopoulos20
}.
We present the proof in Section~\ref{s:anais}, below.

\subsection{Literature review}\label{s:litReview}

Sampling from distributions is a longstanding problem that arises in many applications such as Bayesian inference, statistical Physics and machine learning. 
In many situations of practical interest, the state space~$\mathcal X$ is huge (either finite, but with a computationally intractable size, or a high dimensional manifold).
For any given state~$x \in \mathcal X$, one can typically compute an energy~$U(x)$ measuring how favorable the state is.
Standard models (e.g. the canonical ensemble in statistical physics) dictate that the probability of finding the system in state~$x$ is proportional to~$\tilde \pi_\epsilon(x) = e^{-U(x) / \epsilon}$ (as in~\eqref{e:piDef}), where~$\epsilon > 0$ is a parameter (often the absolute temperature) controlling how fast the system transitions between states.

Practically, the normalization constant~$Z^\epsilon$ (often called the partition function) is a high dimensional integral (or a sum over huge number of states) and is computationally intractable.
Moreover, even if the normalization constant~$Z^\epsilon$ is known, the enormous state space requires the use of algorithms that can deliver samples even though they can only inspect a miniscule fraction of all possible states.
Such algorithms aren't easy to design, or rigorously analyze, and this makes such sampling problems extremely challenging. 

Markov Chain Monte Carlo (MCMC) is a family of algorithms that is often used to address such problems.
These work by simulating a Markov process whose stationary distribution is~$\pi_\epsilon$, and date back to the celebrated Metropolis--Hastings algorithm~\cite{MetropolisRosenbluthEA53,Hastings70}.
In Euclidean space the Langevin dynamics~\eqref{e:langevin} provides a particularly convenient Markov process with stationary distribution~$\pi_\epsilon$, and this is the basis of \emph{Langevin Monte Carlo (LMC)}, \emph{Metropolis Adjusted Langevin Monte Carlo (MALA)} and various other sampling algorithms (see for instance~\cite{Liu08}).

One issue that requires attention when using MCMC based algorithms is the rate at which the Markov chain converges to equilibrium.
If this rate is too slow, it may require simulating the Markov chain for impractical amounts of time before obtaining good samples.
In some cases, one has quantitative estimates on the rate of convergence.
For~\eqref{e:langevin}, it is known that if~$U$ is uniformly convex, then
\begin{equation}
  W_2( \dist(X^\epsilon_t), \pi_\epsilon )
    \leq e^{-C(U) t / \epsilon}
    W_2( \dist(X^\epsilon_0), \pi_\epsilon )
    \,,
\end{equation}
which means that simulating~\eqref{e:langevin} will yield good samples of the Gibbs measure in short time, even when the temperature~$\epsilon$ is small.
Vempala and Wibisono~\cite{VempalaWibisono19} (see also~\cite{Chewi23}) showed that this is also true for the Euler--Maruyama discretization of~\eqref{e:langevin}, making LMC a practically viable sampling algorithm for log concave distributions even in high dimensions.

When~$U$ is not convex, however, the situation is different.
If~$U$ is multi-modal (for instance, if~$\pi_1$ is a Gaussian mixture), then the convergence rate of~\eqref{e:langevin} is exponentially small in~$1/\epsilon$.
That is, it takes time~$e^{O(1/\epsilon)}$ for the distribution of~$X^\epsilon_t$ to become close to~$\pi_\epsilon$.
This phenomenon is known as the Arrhenius law~\cite{Arrhenius89}, and occurs for the following reason.
Since the drift in~\eqref{e:langevin} pulls trajectories towards local minima of~$U$, the noise term in~\eqref{e:langevin} has to go against the drift for an~$O(1)$ amount of time in order for trajectories to transition from the basin of attraction of one local minimum to another.
This happens with exponentially small probability, leading to the Arrhenius law.

Several methods have been designed to improve the rate of convergence.
Tempering methods introduced in~\cite{
  SwendsenWang86,
  Neal96a,
  MarinariParisi92,
  Neal11
}
run a Markov chain on a product of the state space at various temperature levels.
Other methods include ideas based on 
  birth-death~\cite{LuLuEA19},
  optimization~\cite{PompeHolmesEA20},
  diffusion models~\cite{ChehabKorbaEA25},
  warm starts~\cite{KoehlerLeeEA25,LeeSantanaGijzen25}.
Authors have also modified~\eqref{e:langevin} by adding a drift~\cite{
  ReyBelletSpiliopoulos15,
  DamakFrankeEA20,
  ChristieFengEA25},
or modifying the diffusion~\cite{EngquistRenEA24}.
In some situations~\cite{
  WoodardSchmidlerEA09,
  GeLeeEA20,
  Son26
}
polynomial convergence bounds have been rigorously proved.

The methods most closely related to this paper are known as \emph{Sequential Monte Carlo (SMC)} algorithms.
The first such instance was developed to study of the average extension of molecular chains~\cite{HammersleyMorton54,RosenbluthRosenbluth55}.
SMC methods use a sequence of auxiliary distributions~$\nu_1$, \dots, $\nu_K$ so that~$\nu_1$ is easy to sample from, $\nu_K$ is the target distribution, and then move samples between distributions using a reweighting / resampling mechanism.
These methods are hugely popular, and we refer the reader to~\cite{DoucetFreitasEA01,Liu08,ChopinPapaspiliopoulos20,SyedBouchardCoteEA24} for a broad overview.
In the context of multi-modal distributions, rigorous convergence bounds were proved in~\cite{
  Schweizer12,
  PaulinJasraEA19,
  MathewsSchmidler24,
  LeeSantanaGijzen24,
  Han25,
  HanIyerEA26
}.

The algorithms we use in this paper (Algorithms~\ref{a:aisLMC} and~\ref{a:aisAutoNormalized}) are obtained using Neal's \emph{Annealed Importance Sampling (AIS)} algorithm~\cite{Neal01}, combined with Langevin Monte Carlo.
While AIS and Langevin AIS are immensely popular, there are very few rigorous convergence bounds that apply in the setting of Theorem~\ref{t:langevin} and~\ref{t:langevinGen-anais-intro}.
Specifically, in our setting we make no apriori assumption about the mixture decomposition, symmetry, shape of the wells, or apriori assume knowledge of how the mass distribution in wells changes as the temperature varies.
We only assume non-degeneracy of critical points, and regularity of~$U$.
(The assumption that~$U$ is a double well potential with wells of equal depth can be relaxed as mentioned in Remark~\ref{r:unequal}.)
To the best of our knowledge, the only results that apply in this setting are~\cite{HanIyerEA26,Han25,Son26}, and we now comment on the relationship between these results and the present paper.

In~\cite{HanIyerEA26} the authors used an SMC algorithm which resampled points at every level, instead of reweighting them as we do in Algorithms~\ref{a:aisLMC} and~\ref{a:aisAutoNormalized}.
The advantage of this method is that it is keeps the effective sample size constant, and is extremely popular in practice~\cite{ChopinPapaspiliopoulos20}.
The disadvantage is that particles are now only exchangeable, and not independent, and so theoretical bounds harder to obtain.
In~\cite{HanIyerEA26} the authors show the complexity is exactly the same as that of Algorithm~\ref{a:aisAutoNormalized} (given by~\eqref{e:complexityLAAIS}).

In~\cite{Han25} the author shows the same SMC algorithm can be used in the non-compact setting when the state space~$\mathcal X = \R^d$.
The author uses a coupling argument in~\cite{MarionMathewsEA23} and obtains error estimates in probability.
For a fixed error, the time complexity of this algorithm is~$O(\abs{\ln \epsilon}^{10/3}/\epsilon^7)$ as~$\epsilon \to 0$.

Finally in~\cite{Son26} revisits this sampling problem using Metropolis ball walks and parallel tempering.
The author uses a soft domain decomposition~\cite{MadrasRandall02,WoodardSchmidlerEA09} to show that for a fixed error, the time complexity behaves like~$O(1/\epsilon^{11})$ as~$\epsilon \to 0$.

As mentioned earlier (Remark~\ref{r:complexityAis} and~\eqref{e:complexityLAIS}), for a fixed error the time complexity of Langevin AIS (Algorithm~\ref{a:aisLMC}) scales like~$O(1/\epsilon^2)$ as~$\epsilon \to 0$.
The main new contribution of this paper of this paper over~\cite{HanIyerEA26,Han25,Son26} is twofold.
First, the algorithm has smaller complexity ($O(1/\epsilon^2)$, vs~$O(1/\epsilon^4)$ or higher).
Second, the proof identifies a simple and useful quantity, $C_P(T)$, that controls the sample error of AIS (equation~\eqref{e:Ptrk2-ub-intro}, see also Theorem~\ref{t:ais} and Section~\ref{s:chi2}, below) in a general setting.
We presently prove Theorem~\ref{t:langevin} by bounding~$C_P(T)$ using spectral estimates from~\cite{
  Kolokoltsov00,
  BovierGayrardEA05,
  MenzSchlichting14,
  HanIyerEA26}.
There may be room to bound~$C_P(T)$ using different techniques, bypassing the limitation of spectral methods, but this goes beyond the scope of the present paper and is left for future study.

\subsection*{Plan of this paper}

In Section~\ref{s:ais} we state two results (Proposition~\ref{p:expectedDiff} and Theorem~\ref{t:ais}) addressing convergence of AIS in a general setting, and control the sampling error in terms of~$C_P(T)$.
In Section~\ref{s:aisConv} we study Langevin AIS and its autonormalized version, and state generalizations of Theorems~\ref{t:langevin} and~\ref{t:langevinGen-anais-intro}.
In Section~\ref{s:aisProof} we prove Proposition~\ref{p:expectedDiff} and Theorem~\ref{t:ais}.
In Section~\ref{s:langevin} we prove the generalization of Theorem~\ref{t:langevin} (Theorem~\ref{t:langevinGen}), and finally in Section~\ref{s:anais} we prove the generalization of Theorem~\ref{t:langevinGen-anais-intro} (Theorem~\ref{t:langevinGen-anais}).

\section{Convergence results for AIS}\label{s:ais}

As mentioned earlier, we prove Theorem~\ref{t:langevin} by showing that the quantity~$C_P(T)$ in~\eqref{e:Ptrk2-ub-intro} can be used to control the sampling error of AIS in a general setting.
We begin by stating this precisely.

\subsection{Bias and Variance estimates for AIS}\label{s:meanVariance}
Suppose~$\pi = \tilde \pi / Z$ is a target distribution from which samples are desired.
Here~$\tilde \pi$ is an unnormalized probability distribution which is easy to compute, and~$Z = \int_{\mathcal X} \tilde \pi$ is the normalization constant (which is hard to compute).
The AIS algorithm~\cite{Neal01} uses an auxiliary family of distributions with densities proportional to~$\tilde \pi_1$, \dots, $\tilde \pi_{K+1} = \tilde \pi$ (called a \emph{tempering sequence}), along with reversible Markov transition kernels~$P_1$, \dots, $P_K$.
The algorithm now successively simulates a Markov chain with kernel~$P_k$, and then reweights the points using the ration~$\tilde \pi_{k+1} / \tilde \pi_k$ and is described precisely as Algorithm~\ref{a:ais2}.
\begin{algorithm}[htb]
  \caption{Annealed Importance Sampling (AIS)}\label{a:ais2}
  \begin{algorithmic}[1]
    \item[\textbf{Requirements:}]
    \Statex
      \begin{enumerate}[(1)]
	\item Tempering sequence of unnormalized densities~$\tilde \pi_1$, \dots, $\tilde \pi_K$.
	\item Corresponding reversible Markov transition kernels~$P_1$, \dots, $P_K$.
	\item Running time~$T$.
      \end{enumerate}
    \State Start with~$w_0 = 1$, $X_0$ arbitrary.
    \For{$k=1, \dots, K$}
      \State Sample~$X_k$ from  $P_{k}^T(X_{k-1}, \cdot)$
      \State $\tilde w_{k} \gets \tilde w_{k-1} \frac{\tilde \pi_{k+1}(X_k)}{\tilde \pi_{k}(X_{k})}$.
	
    \EndFor
    \State \Return~$X_{K}$ and the unnormalized weight~$\tilde w_{K}$.
  \end{algorithmic}
\end{algorithm}

The goal of this section is to obtain a quantitative error bound on AIS, provided the tempering sequence satisfies the following assumptions.

\begin{assumption}\label{a:reversible}
  For each~$k \in \set{1, \dots, K+1}$ we have a reversible Markov transition kernel~$P_k$ whose stationary distribution is~$\pi_k = \tilde \pi_k / Z_k$, where~$Z_k = \int_{\mathcal X} \tilde \pi_k$.
\end{assumption}

\begin{assumption}\label{a:r2}
There exists constants~$T_0$, $C_w$ such that for every~$T \geq T_0$ we have
  \begin{equation}\label{e:Ptrk2-ub}
    \prod_{k = 1}^K
      \norm{P_{k}^T r_{k}^2}_{L^\infty} \leq C_w
      \,,
    \quad\text{where}\quad
    r_{k} \defeq \frac{\pi_{k+1}}{\pi_{k}}
    \,.
  \end{equation}
\end{assumption}

\begin{assumption}\label{a:L1Linf}
  The uniform mixing time of~$P_1$ is finite.
  Explicitly,
  \begin{equation}\label{e:umixDef}
    t_{\mix, 1}^\infty \defeq \min \set[\bigg]{n \in \N \st \sup_{x \in \mathcal X} \norm[\bigg]{ \frac{P_1^n(x, \cdot)}{\pi_1(\cdot)} - 1 }_\infty < \frac{1}{2} }
    < \infty
  \end{equation}
\end{assumption}

We will now show that if we use AIS with a tempering sequence satisfying Assumptions~\ref{a:reversible}--\ref{a:L1Linf}, then the bias decreases exponentially with~$T/t^\infty_{\mix, 1}$, and variance is bounded by $4 C_w$.

\begin{proposition}[Bias and variance bounds]\label{p:expectedDiff}
  Suppose~$\set{P_k}_{1 \leq k \leq K}$ are Markov transition kernels satisfying Assumptions~\ref{a:reversible}--\ref{a:L1Linf}.
  Let~$X_K$, $\tilde w_K$ be obtained from Algorithm~\ref{a:ais2} with~$K+1$ levels, running time~$T$, and kernels~$\set{P_k}$.
  For all bounded measurable functions~$f$, we have the bias and variance estimates
  \begin{gather}
    \label{e:normalized-bias}
    \abs[\bigg]{\frac{\E \tilde w_{K} f(X_{K})}{\E \tilde w_{K}} - \ip{f, \pi_{K+1}}}
      \leq 2^{2- T / t_{\mix, 1}^\infty} \norm{f}_\infty\,,
      \quad\text{provided} \quad T \geq t_{\mix, 1}^\infty
    \\
    \label{e:variance-estimate}
      \var\paren[\bigg]{\frac{\tilde w_{K} f(X_{K})}{\E \tilde w_{K}}}
	\leq 4 C_w \norm{f}_\infty^2
	\,,
      \quad\text{provided}\quad T \geq T_0
      \,.
  \end{gather}
\end{proposition}

In practice, one would take~$N$ independent samples from Algorithm~\ref{a:ais2}, and estimate~$\E \tilde w_K$ using the empirical mean.
An immediate consequence of Proposition~\ref{p:expectedDiff} is a quantitative bound on the convergence of the empirical measure.

\begin{theorem}[Empirical measure convergence]\label{t:ais}
  Suppose~$\set{P_k}_{1 \leq k \leq K}$ are Markov transition kernels satisfying Assumptions~\ref{a:reversible}--\ref{a:L1Linf}.
  Given~$\delta > 0$, choose
  \begin{equation}\label{e:NT-choice}
    N = \frac{64C_w}{\delta^2}
    \quad\text{and}\quad
    T = \max\set*{T_0, (4+\abs{\log_2 \delta})t_{\mix, 1}^\infty}\,.
\end{equation}
  Let~$\set{X^i_K, \tilde w^i_K}_{1 \leq i \leq N}$ be~$N$ independent realizations of the points and weights returned by Algorithm~\ref{a:ais2}, with~$K+1$ levels, running time~$T$, and kernels~$\set{P_k}$.
  For every bounded measurable $f$, we have
  \begin{equation}\label{e:ais-small-error-and-wki-def}
    \norm[\Big]{
      \sum_{i=1}^N w^i_K f(X^i_K) - \ip{f, \pi_{K+1}}
    }_{L^2(\P)}
    \leq \norm{f}_\infty \delta\,,
    \quad\text{where}\quad
    w^i_K \defeq \frac{\tilde w^i_K}{\sum_{j = 1} \tilde w^j_K}
    \,.
  \end{equation}
  Moreover, if~$\mu_{K+1, N}$ is the empirical measure defined by
  \begin{equation}
    \mu_{K+1, N} \defeq \sum_{i=1}^N w^i_K \delta_{X^i_K}\,,
  \end{equation}
  then for every~$s > d/2$
  \begin{equation}\label{e:empHs}
    \E\norm[\big]{\mu_{K+1, N} - \pi_{K+1}}_{H^{-s}}^2
      \leq C_s \delta^2
      \,.
  \end{equation}
\end{theorem}

\begin{remark}[Time complexity]\label{r:timeComplexity}
  In order to obtain samples that satisfy~\eqref{e:ais-small-error-and-wki-def}, one has to simulate~$N$ realizations of~$T$ iterations of each of the chains~$P_1$, \dots, $P_K$, and so the time complexity is
  \begin{equation}\label{e:complexity}
    O(N T K) = \frac{64 K C_w t_{\mix, 1}^\infty}{\delta^2} \paren[\big]{4 + \abs{\log_2 \delta} }
  \end{equation}
\end{remark}

As mentioned earlier, when working with weighted samples it is important to estimate the \emph{effective sample size} (see~\eqref{e:ess}) and ensure that the weights don't concentrate on a few points.
This can be done quickly from the variance bound~\eqref{e:variance-estimate}.

\begin{proposition}[Effective sample size]\label{p:essBound}
  Let~$w^1_K$, \dots, $w^N_K$ be the normalized weights defined in~\eqref{e:ais-small-error-and-wki-def}.
  Then the sum of the squared weights satisfies
  \begin{equation}\label{e:ais-ssw-ub}
    \E \sum_{i = 1}^N (w^i_K)^2 
      \leq \frac{8\paren*{4C_w +1}}{N}
    \,.
  \end{equation}
  Consequently, the expected effective sample size is bounded by
  \begin{equation}\label{e:ais-ess-lb}
    \E \ESS(w^1_K, \dots, w^N_K)
      \geq \frac{N}{8\paren*{4C_w +1}}
    \,.
  \end{equation}
\end{proposition}

Before delving into the proofs, we now briefly discuss the assumptions and implications of Proposition~\ref{t:ais}.

\subsection{Relation to \texorpdfstring{$\chi^2$}{chi2}-divergence}\label{s:chi2}
To understand the assumptions better, suppose momentarily we were able to choose~$T = \infty$ in~\eqref{e:Ptrk2-ub}.
In this case, the left hand side of~\eqref{e:Ptrk2-ub} better can be bounded in terms of the~$\chi^2$-divergence between each of the intermediate distributions.

\begin{proposition}\label{p:aisChi2}
  If~$T = \infty$ then the left hand side of~\eqref{e:Ptrk2-ub} satisfies
  \begin{equation}\label{e:prodPkInfChi2}
    \prod_{k = 1}^K \norm{P_{k}^\infty r_{k}^2}_{L^\infty}
      \leq \exp\paren[\bigg]{
	  \sum_{k = 1}^K \chi^2(\pi_{k+1}; \pi_{k})
	}
  \end{equation}
\end{proposition}

Recall the~$\chi^2$-divergence appearing above is a commonly used measure of the difference between two distributions.
Explicitly, for two distributions with densities~$p, q$ the~$\chi^2$-divergence is defined by
\begin{equation}
  \chi^2(p; q) \defeq \ip[\Big]{ \frac{p^2}{q^2} - 1, q}
  \,.
\end{equation}

Let us now momentarily assume that~$\pi_k$'s are all Gaussian, with~$\pi_1 \sim \mathcal N(0, I_d)$ and~$\pi_{K+1} \sim \mathcal N(0, \epsilon I_d)$.
If we choose~$K = 1$ then AIS reduces to a vanilla importance sampling, and one can explicitly compute (see for instance Proposition~17.1 in~\cite{ChopinPapaspiliopoulos20}) that the right hand side of~\eqref{e:prodPkInfChi2} is~$O(1 / \epsilon^{d/2})$.
This is comparable with a standard rejection sampling cost, and is too large to be practical.

However, if we instead choose~$K = d/\epsilon$, then the right hand side of~\eqref{e:prodPkInfChi2} can be bounded \emph{independent} of both~$\epsilon$ and~$d$.
Thus, using AIS here with~$O(d/\epsilon)$ intermediate levels gives samples with variance~$O(1)$, which is a huge improvement.

\begin{proposition}\label{p:aisGaussian}
  If the target distribution~$\pi$ is the Gaussian~$\mathcal N(0, \epsilon I_d)$, then choosing
  \begin{equation}\label{e:GaussianKChoice}
    K = \ceil[\bigg]{\frac{d}{\epsilon}}
    \,,
  \end{equation}
  choosing the temperatures~$\set{\epsilon_k}$ so that~$\epsilon_{K+1} = \epsilon$ and~$\set{1/\epsilon_k}_{1 \leq k \leq K}$ are linearly spaced, and choosing~$\pi_k$ to be the Gaussian~$\mathcal N(0, \epsilon_k I_d)$ will ensure
  \begin{equation}\label{e:prodPkInfGaussian}
    \prod_{k = 1}^K \norm{P_{k}^\infty r_{k}^2}_{L^\infty}
      \leq C
      \,.
  \end{equation}
  In this case, Proposition~\ref{p:expectedDiff} with~$T = \infty$ will yield a variance bound that is \emph{independent} of both~$\epsilon$ and~$d$.
\end{proposition}

The proof of Proposition~\ref{p:aisGaussian} is elementary, and presented in Section~\ref{s:chi2Proof}, below.

\section{Convergence results for Langevin AIS}\label{s:aisConv}

\subsection{Assumptions}
In order to state the assumptions in Theorem~\ref{t:langevin} precisely, we first describe the assumptions that are required on the potential~$U$.
In short, we need the potential to be a regular, double-well function with wells of \emph{nearly equal} depth.
The criterion that the wells have nearly equal depth is required for the multimodal sampling problem to be non-degenerate in the sense that the mass in each well remains bounded away from~$0$.
These assumptions are the same as the assumptions in~\cite{HanIyerEA26}, which are explained in detail in~\cite[Section 4.1]{HanIyerEA26}.
We quote them here for easy reference, and refer the reader to~\cite{HanIyerEA26} for a detailed explanation and motivation.

\begin{assumption}\label{a:criticalpts}
  The function~$U\in C^{6 \varmax (1 + \floor{d/2})}(\mathbb T^d,\mathbb R)$, has a nondegenerate Hessian at all critical points, and has exactly two local minima located at~$x_{\min, 1}$ and~$x_{\min, 2}$.
  We normalize~$U$ so that
\begin{equation}\label{e:Upositive}
  0 = U(x_{\min, 1} )  \leq U( x_{\min, 2} )
  .
\end{equation}
\end{assumption}

Define the saddle height between~$x_{\min,1}$ and $x_{\min,2}$ to be the minimum amount of energy needed to go from the global minimum~$x_{\min, 1}$ to~$x_{\min, 2}$.
Explicitly, the saddle height is
\begin{equation}\label{e:UHatDef}
    \hat{U} = \hat{U}(x_{\min,1},x_{\min,2}) \defeq \inf_\omega \sup\limits_{t\in [0,1]}U(\omega(t))
    .
\end{equation}
Here the infimum above is taken over all continuous paths~$\omega \in C( [0, 1]; \mathbb T^d)$ such that~$\omega(0) =x_{\min,1}$, $\omega(1) = x_{\min,2}$.

\begin{assumption}\label{assumption: nondegeneracy}
   The saddle height between $x_{\min,1}$ and $x_{\min,2}$ is attained at a unique critical point $s_{1,2}$ of index one.
   That is, 
   the first eigenvalue of $\Hess U(s_{1,2})$ is negative and the others are positive.
\end{assumption}
The \emph{energy barrier}, denoted by~$\hat{\gamma}$, is defined to be the minimum amount of energy needed to go from the (possibly local) minimum~$x_{\min, 2}$ to the global minimum~$x_{\min, 1}$.
In terms of~$s_{1, 2}$, the energy barrier~$\hat \gamma$ and the saddle height are given by
\begin{equation}\label{e:gammaHatDef}
    \hat{\gamma}\defeq U(s_{1,2})-U(x_{\min,2}),
    \quad\text{and}\quad
    \hat U = U(s_{1,2})\,.
\end{equation}
The ratio $\hat{\gamma}_r$ is the ratio of the saddle height~$\hat U$ to the energy barrier~$\hat \gamma$, given by
\begin{equation}\label{e:gammaHatRDef}
    \hat \gamma_r\defeq \frac{\hat{U}}{\hat{\gamma}}\,.
\end{equation}
We recall the basin of attraction around~$x_{\min, i}$, denoted by~$\Omega_i$, is defined by
\begin{equation}
  \Omega_i
    \defeq
    \set[\Big]{
      y\in\mathbb T^d \st
      \lim_{t\to\infty}y_t=x_{\min,i}, \text{ where }\dot{y}_t=-\nabla U(y_t)
	\text{ with } y_0=y
      }\,.
\end{equation}

\begin{assumption}\label{a:massRatioBound}
  There exists~$0 \leq \epsilon_{\min} < \epsilon_{\max} \leq \infty$, a constant~$C_m$ such that
   \begin{equation}\label{e:massRatioBound}
     \inf_{\substack{
       \epsilon\in [\epsilon_{\min},\epsilon_{\max}]\\
       0 < \epsilon < \infty
     }}
      \pi_{\epsilon}(\Omega_i)
	\geq \frac{1}{C_m^2}
	.
   \end{equation}
\end{assumption}

\begin{remark}
  For simplicity, we have assumed that~$U$ has only two wells.
  If~$U$ has more than two wells, the techniques we use will still apply provided we impose a non-degeneracy condition on~$U$.
  The precise assumption is stated in Section 10 in~\cite{HanIyerEA26}, and the required modifications to the proof are straightforward.
\end{remark}

\subsection{Convergence results for Langevin AIS}

In this section we precisely state the main result of this paper (Theorem~\ref{t:langevinGen}, below, which is a generalization of Theorem~\ref{t:langevin}).
We begin by stating bias and variance bounds for Algorithm~\ref{a:aisLMC}.

\begin{proposition}[Bias and variance bounds]\label{p:langevinGen}
  Let~$U$ be a double well potential that satisfies Assumptions~\ref{a:criticalpts}--\ref{a:massRatioBound}, and let~$\epsilon_1 \in (\epsilon_{\min}, \epsilon_{\max}]$.
  Let~$\nu > 0$ be a fixed constant.
  There exists constants~$\bar C_w = \bar C_w(U, \nu)$, and~$C_T = C_T(U)$ such that the following holds.
  For any~$\epsilon \in (\epsilon_{\min}, \epsilon_{1}]$ choose
  \begin{equation}\label{e:KTchoice}
    K = \ceil[\bigg]{ \frac{1}{\epsilon \nu} }
    \quad\text{and}\quad
    T > \max\set[\Big]{
      t_{\mix, \epsilon_1}^\infty,
      C_T \paren[\Big]{\frac{1}{\epsilon} + \log K}
    }
    \,,
  \end{equation}
  and choose~$\epsilon_2$, \dots, $\epsilon_{K+1} = \epsilon$ so that~$\set{1/\epsilon_k}_{1 \leq k \leq K+1}$ are linearly spaced.
  Run Algorithm~\ref{a:aisLMC} with this choice of parameters and obtain the point~$X_K$, and unnormalized weight~$\tilde w_K$.
  Then, for all bounded measurable test functions~$f$, we have the bias and variance estimates
  \begin{align}\label{e:normalized-bias-langevin}
    \abs[\bigg]{\frac{\E \tilde w_{K} f(X_{K})}{\E \tilde w_{K}} - \ip{f, \pi_{\epsilon}}}
      &\leq 2^{2- T / t_{\mix, 1}^\infty} \norm{f}_\infty
    \,,
    \\
    \label{e:variance-estimate-langevin}
	\var\paren[\bigg]{\frac{\tilde w_{K} f(X_{K})}{\E \tilde w_{K}}}
	  &\leq 4 \bar C_w \norm{f}_\infty^2
	  \,.
  \end{align}
\end{proposition}
\begin{remark}
  We will shortly see that~$t_{\mix, \epsilon}^\infty < \infty$ for every~$\epsilon > 0$, but grows exponentially with~$1/\epsilon$ as~$\epsilon \to 0$.
  Choosing~$T > t_{\mix, \epsilon}^\infty$ is of course impractical when~$\epsilon$ is small, however, the choice of~$T$ in~\eqref{e:KTchoice} only requires~$T > t_{\mix, \epsilon_1}^\infty$, which is practically tractable when~$\epsilon_1$ is large.
\end{remark}

Proposition~\ref{p:langevinGen} allows us to bound the error when we perform repeated independent runs of Algorithm~\ref{a:aisLMC}.
This is the main result of this paper.

\begin{theorem}\label{t:langevinGen}
  Let~$U$ be a double well potential that satisfies Assumptions~\ref{a:criticalpts}--\ref{a:massRatioBound}, and let~$\epsilon_1 \in (\epsilon_{\min}, \epsilon_{\max}]$.
  Given~$\nu > 0$, let~$\bar C_w$, $C_T$ be as in Proposition~\ref{p:langevinGen}.
  For any~$\delta > 0$ and~$\epsilon \in (\epsilon_{\min}, \epsilon_1]$ choose
  \begin{align}
    \label{e:K-choice-emp}
    K &= \ceil[\bigg]{ \frac{1}{\epsilon \nu} }
    \,,
    \\
    \label{e:T-choice-emp}
    T &> \max\set[\Big]{
      C_T \paren[\Big]{\frac{1}{\epsilon} + \log K},
      (4+\abs{\log_2 \delta})t_{\mix, \epsilon_1}^\infty
    }\,,
    \\
    \label{e:N-choice-emp}
    N &= \ceil[\bigg]{\frac{64 \bar C_w}{\delta^2}}
    \,,
  \end{align}
  and choose~$\epsilon_2$, \dots, $\epsilon_{K+1} = \epsilon$ so that~$\set{1/\epsilon_k}_{1 \leq k \leq K+1}$ are linearly spaced.
  Perform~$N$ independent runs of Algorithm~\ref{a:aisLMC} with these parameters and let~$(X^1_K, \tilde w^1_K)$, \dots, $(X^N_K, \tilde w^N_K)$ be the resulting points and unnormalized weights.
  Define the empirical measure~$\mu_N$ by~\eqref{e:muN-def}.
  Then for every bounded test function~$f$ we have~\eqref{e:emp-error-bd}.
  Consequently, for every~$s > d/2$ there exists an explicit dimensional constant~$C_s$ (independent of~$\epsilon_1, \epsilon, \nu, \delta$) such that~\eqref{e:emp-HsBd} holds.
\end{theorem}

\begin{remark}[Effective sample size]\label{r:essLangevin}
  Proposition~\ref{p:essBound} and~\eqref{e:variance-estimate-langevin} immediately show that the sum of the square of the normalized weights~$w_i$, \dots, $w_i$ (defined in~\eqref{e:muN-def}) is bounded above by
  \begin{equation}\label{e:sswLangevin}
    \E \paren[\Big]{\sum_{i=1}^N w_i^2 } \leq \frac{8\paren*{4\bar C_w +1}}{N}\,.
  \end{equation}
  Hence the effective sample size is bounded below by
  \begin{equation}\label{e:ESSLangevin}
    \E \ESS(w_1, \dots, w_N) \geq \frac{N}{8\paren*{4\bar C_w +1}}\,.
  \end{equation}
\end{remark}

The proofs of Proposition~\ref{p:langevinGen} and Theorem~\ref{t:langevinGen} are in Section~\ref{s:langevin}.

\subsection{Convergence results for autonormalized Langevin AIS}
We now consider Algorithm~\ref{a:aisAutoNormalized}, which is an auto-normalized version of Algorithm~\ref{a:aisLMC}, where we re-normalize the weights at every step.
This is a generalization of Theorem~\ref{t:langevinGen-anais-intro} to more general potentials.

\begin{theorem}\label{t:langevinGen-anais}
  Let~$U$ be a double well potential that satisfies Assumptions~\ref{a:criticalpts}--\ref{a:massRatioBound}, and let~$\epsilon_1 \in (\epsilon_{\min}, \epsilon_{\max}]$.
  Given~$\alpha, \delta, \nu > 0$, there exists constants~$C_N(\nu, U)$ and~$\hat C_T(\alpha, \nu, U)$ such that the following holds.
  For any~$\delta > 0$ and~$\epsilon \in (\epsilon_{\min}, \epsilon_1]$, choose
  \begin{align}
    \label{e:K-choice-emp-anais}\noeqref{e:K-choice-emp-anais}
    K &= \ceil[\bigg]{ \frac{1}{\epsilon \nu} }
    \,,
    \\
    \label{e:T-choice-emp-anais}
    T &>  \hat C_T\paren*{K^{(1+\alpha)\hat \gamma_r} + \frac{1}{\epsilon} + \log\paren*{\frac{1}{\delta}} + \log\paren*{N}}
    \,,
    \\
    \label{e:N-choice-emp-anais}
    N &= \ceil[\bigg]{\frac{C_N}{\delta^2}}K^2
    \,,
  \end{align}
  and choose~$\epsilon_2$, \dots, $\epsilon_{K+1} = \epsilon$ so that~$\set{1/\epsilon_k}_{1 \leq k \leq K+1}$ are linearly spaced.
  (The constant $\hat \gamma_r$ in~\eqref{e:T-choice-emp-anais} is defined in~\eqref{e:gammaHatRDef}.)
 Let~$\set{X^i_K, \tilde w^i_K}_{1 \leq i \leq N}$ be the~$N$ points and normalized weights returned by Algorithm~\ref{a:aisAutoNormalized}, with~$K+1$ levels, running time~$T$, and kernels~$\set{P_k}$.
  For every bounded measurable $f$, we have
  \begin{equation}\label{e:an-ais-error-bound}
    \norm[\Big]{
      \sum_{i=1}^N w^i_K f(X^i_K) - \ip{f, \pi_{K+1}}
    }_{L^2(\P)}
    \leq \norm{f}_\infty \delta\,,
    \,.
  \end{equation}
  Moreover, if~$\mu_{K+1, N}$ is the empirical measure defined by
  \begin{equation}
    \mu_{K+1, N} \defeq \sum_{i=1}^N w^i_K \delta_{X^i_K}\,,
  \end{equation}
  then for every~$s > d/2$,                              
  \begin{equation}\label{e:empHs-anais}
    \E\norm[\big]{\mu_{K+1, N} - \pi_{K+1}}_{H^{-s}}^2
      \leq C_s \delta^2
      \,.
  \end{equation}
\end{theorem}

As before, it is important to study the effective sample size.
For AIS (Theorems~\ref{t:ais} and~\ref{t:langevinGen}) the bound for the effective sample size followed directly from the convergence bound (see Proposition~\ref{p:essBound} and Remark~\ref{r:essLangevin}).
For the autonormalized version, we first need to first bound the effective sample size in order to prove the convergence bound in Theorem~\ref{t:langevinGen-anais}.
We state this as our next result.

\begin{proposition}[Effective sample size]\label{p:essANLangevin}
  Using the same assumptions and notation as Theorem~\ref{t:langevinGen-anais}, there exists a constant~$C_1 = C_1(\nu, \epsilon_1)$ such that for any $k\in \set{1, \ldots, K}$ we have
\begin{equation}\label{e:sum-w-k1-to-k}
\E\paren*{\sum_{i=1}^N w_{k,i}^2} \leq \frac{C_1\bar C_w^2}{N}\,,
\end{equation}
  where~$\bar C_w$ is the constant from Proposition~\ref{p:langevinGen}.
Consequently, the effective sample size satisfies the lower bound
  \begin{equation}\label{e:essAN}
    \E \ESS(w_{k, 1}, \dots, w_{k, N}) \geq \frac{N}{C_1\bar C_w^2}\,.
  \end{equation}
\end{proposition}

The proofs of Theorem~\ref{t:langevinGen} and Proposition~\ref{p:essANLangevin} are presented in Section~\ref{s:anais}.

\section{Proof of convergence for AIS}\label{s:aisProof}

\subsection{Bias and Variance estimates (Proposition~\ref{p:expectedDiff})}

To prove the bias estimate~\eqref{e:normalized-bias} we first need an estimate on the unnormalized bias.
\begin{lemma}\label{l:bias_M0}
  If Assumptions~\ref{a:reversible} and~\ref{a:L1Linf} hold, then for any bounded measurable function~$f$, we have
\begin{equation}\label{e:bias-first-level}
\abs*{\E\left[\tilde w_K f(X_K)\right]
- \frac{Z_{K+1}}{Z_1}\ip{f,\pi_{K+1}}}
\le
2^{-T/t_{\mix,1}^\infty}
\frac{Z_{K+1}}{Z_1}
\ip{\abs{f},\pi_{K+1}} .
\end{equation}
\end{lemma}
\begin{proof}
For notational convenience, define
\begin{equation}\label{e:bias-def}
\Bias_{i}(f)  \defeq 
\abs*{\E\left[\tilde w_i f(X_i)\right]
- \frac{Z_{i+1}}{Z_1}\ip{f,\pi_{i+1}}}\,.
\end{equation}
Notice
\[
\Bias_{K}(f) \leq \Bias_{K}(f^+) + \Bias_{K}(f^-),
\]
and hence it suffices to establish~\eqref{e:bias-first-level} for non-negative functions~$f$.  
Thus, without loss of generality, we subsequently assume that $f \ge 0$.
\medskip

Observe that the mean of the estimator can be rewritten as
\begin{align}\label{e:mean-estimator}
\E\brak*{\tilde w_K f(X_K)} 
&= \E\brak*{\tilde w_{K-1} (\tilde r_K f)(X_K)} \\
&= \E\brak*{\tilde w_{K-1} P_K^T (\tilde r_K f)(X_{K-1})}\,.
\end{align}
On the other hand, using reversibility, the corresponding true mean satisfies
\begin{equation}\label{e:true-mean-change}
\frac{Z_K}{Z_1} 
\ip{P_K^T (\tilde r_K f), \pi_K}
= \frac{Z_K}{Z_1} \ip{\tilde r_K f, \pi_K}
= \frac{Z_{K+1}}{Z_1} \ip{f, \pi_{K+1}}\,.
\end{equation}

Combining~\eqref{e:mean-estimator}, \eqref{e:true-mean-change}, and the definition of the bias in~\eqref{e:bias-def}, we obtain
\[
\Bias_{K}(f) 
= \Bias_{K-1}\!\paren*{P_K^T (\tilde r_K f)}.
\]
Iterating this identity down to level~$1$ yields
\begin{equation}\label{e:biasM0-bias20}
\Bias_{K}(f) = \Bias_{1}\!\paren*{\tilde f_2},
\end{equation}
where
\[
\tilde f_i \defeq S_i S_{i+1}\cdots S_K f,
\qquad
S_k h \defeq P_k^T (\tilde r_k h).
\]
\medskip

Next, observe that
\[
\E\brak*{\tilde w_1 \tilde f_2(X_1)}
= P_1^T(\tilde r_1 \tilde f_2)(X_0),
\qquad
\ip{\tilde r_1 \tilde f_2, \pi_1}
= \frac{Z_2}{Z_1}\ip{\tilde f_2,\pi_2}.
\]
Since~$f \geq 0$, positivity of~$P_k$ implies~$\tilde f_2 \geq 0$.
Now Assumption~\ref{a:L1Linf} implies
\begin{align}\label{e:first-level-mixing}
\Bias_1(\tilde f_2)
  &=
  \abs{P_1^T(\tilde r_1 \tilde f_2)(X_0) - \ip{\tilde r_1 \tilde f_2, \pi_1}}
\\
  &\leq
2^{-T/t_{\mix,1}^\infty}
\ip{\tilde r_1 \tilde f_2,\pi_1}
=
2^{-T/t_{\mix,1}^\infty}
\frac{Z_2}{Z_1}
\ip{\tilde f_2,\pi_2}.
\end{align}
Moreover,
\begin{align}\label{e:restore-mean}
\frac{Z_2}{Z_1}\ip{\tilde f_2,\pi_2}
&=
\frac{Z_2}{Z_1}
\ip{P_2^T(\tilde r_2 \tilde f_3),\pi_2}
=
\frac{Z_3}{Z_1}\ip{\tilde f_3,\pi_3} \nonumber\\
&=\cdots=
\frac{Z_K}{Z_1}\ip{\tilde f_K,\pi_K}
=
\frac{Z_{K+1}}{Z_1}\ip{f,\pi_{K+1}} .
\end{align}
Combining~\eqref{e:biasM0-bias20}, \eqref{e:first-level-mixing}, and~\eqref{e:restore-mean} yields~\eqref{e:bias-first-level}, which completes the proof.
\end{proof}

Lemma~\ref{l:bias_M0} and Assumption~\ref{a:r2} quickly yield Proposition~\ref{p:expectedDiff}.

\begin{proof}[Proof of Proposition~\ref{p:expectedDiff}]
Applying~\eqref{e:bias-first-level} with the test function $f \equiv 1$, we obtain
\begin{equation}\label{e:wktilde-bias}
\abs*{\E \tilde w_K - \frac{Z_{K+1}}{Z_1}}
\le
2^{-T/t_{\mix,1}^\infty}
\frac{Z_{K+1}}{Z_1}.
\end{equation}

Consequently, if $T \ge t_{\mix,1}^\infty$, then
\begin{equation}\label{e:wktilde-lb}
\E \tilde w_K
\ge
\frac{1}{2}\frac{Z_{K+1}}{Z_1}.
\end{equation}

Next, we estimate 
\begin{align}
&\abs*{
\E\left[\tilde w_K f(X_K)\right]
- \E \tilde w_K \ip{f,\pi_{K+1}}
}
 \\
&\quad \le
\abs*{
\E\left[\tilde w_K f(X_K)\right]
- \frac{Z_{K+1}}{Z_1}\ip{f,\pi_{K+1}}
}
 \\
&\qquad
+
\abs*{
\frac{Z_{K+1}}{Z_1}\ip{f,\pi_{K+1}}
- \E \tilde w_K \ip{f,\pi_{K+1}}
}
 \\
&\quad \overset{\eqref{e:bias-first-level}}{\le}
2^{-T/t_{\mix,1}^\infty}
\frac{Z_{K+1}}{Z_1}
\ip{\abs{f},\pi_{K+1}}
+
\abs*{\E \tilde w_K - \frac{Z_{K+1}}{Z_1}}
\abs*{\ip{f,\pi_{K+1}}}
\\
\label{e:unnormalized-bias}
&\quad \overset{\eqref{e:wktilde-bias}}{\le}
2^{\,1 - T/t_{\mix,1}^\infty}
\frac{Z_{K+1}}{Z_1}
\ip{\abs{f},\pi_{K+1}} .
\end{align}

Combining~\eqref{e:wktilde-lb} and~\eqref{e:unnormalized-bias} yields the desired normalized bias estimate~\eqref{e:normalized-bias}.

\medskip

We now turn to the variance estimate~\eqref{e:variance-estimate}.  
It suffices to prove the bound
\begin{equation}\label{e:unnormalized-variance-estimate}
\var\!\bigl(\tilde w_K f(X_K)\bigr)
\le
\left(\frac{Z_{K+1}}{Z_1}\right)^2
C_w \norm{f}_\infty^2 .
\end{equation}
Indeed, combining this with~\eqref{e:wktilde-lb} immediately implies~\eqref{e:variance-estimate}.

To prove~\eqref{e:unnormalized-variance-estimate}, observe that
\begin{align}
\var\!\bigl(\tilde w_K f(X_K)\bigr)
&\le
\E\left[\tilde w_K^2 f(X_K)^2\right]
\le
\E \tilde w_K^2 \norm{f}_\infty^2
\nonumber \\
&=
\E\left[\tilde w_{K-1}^2 \tilde r_K(X_K)^2\right]
\norm{f}_\infty^2
\nonumber \\
&=
\E\left[
\tilde w_{K-1}^2
P_K^T(\tilde r_K^2)(X_{K-1})
\right]
\norm{f}_\infty^2
\nonumber \\
&\le
\E \tilde w_{K-1}^2
\,
\norm{P_K^T \tilde r_K^2}_\infty
\norm{f}_\infty^2 .
\end{align}

Iterating this argument over $k = 1, \dots, K$ yields
\begin{equation}\label{e:unnormalized-variance-prod-ub}
\var\!\bigl(\tilde w_K f(X_K)\bigr)
\le
\left(
\prod_{k=1}^K
\norm{P_k^T \tilde r_k^2}_\infty
\right)
\norm{f}_\infty^2 .
\end{equation}

Finally, we note that
\begin{equation}
\prod_{k=1}^K
\norm{P_k^T \tilde r_k^2}_\infty
=
\left(\frac{Z_{K+1}}{Z_1}\right)^2
\prod_{k=1}^K
\norm{P_k^T r_k^2}_\infty
\overset{\eqref{e:Ptrk2-ub}}{\le}
C_w
\left(\frac{Z_{K+1}}{Z_1}\right)^2 .
\end{equation}

Combining this with~\eqref{e:unnormalized-variance-prod-ub} proves~\eqref{e:unnormalized-variance-estimate}, and hence~\eqref{e:variance-estimate}. This completes the proof.
\end{proof}

\subsection{Convergence of the empirical measure (Theorem~\ref{t:ais})}

We now use Proposition~\ref{p:expectedDiff} to prove Theorem~\ref{t:ais}.

\begin{proof}[Proof of Theorem~\ref{t:ais}]
We first prove~\eqref{e:ais-small-error-and-wki-def}.  
Define
\begin{equation}\label{e:R-def}
R \defeq \frac{1}{N} \sum_{i=1}^N \tilde w_K^i .
\end{equation}
From the definition of $w_K^i$ in~\eqref{e:ais-small-error-and-wki-def}, we have
\[
\frac{1}{N} \tilde w_K^i = R\, w_K^i.
\]
Adding and subtracting the same quantity and applying the triangle inequality yields, for any $x_i \in \R^d$,
\begin{equation}\label{e:break-to-I1-I2}
\Bigl| \sum_{i=1}^N w_K^i f(x_i) - \ip{f,\pi_{K+1}} \Bigr|
\le I_1 + I_2,
\end{equation}
where
\begin{align}
I_1 &\defeq
\Bigl| \sum_{i=1}^N w_K^i f(x_i)
- \frac{1}{\E \tilde w_K^1}
  \sum_{i=1}^N R w_K^i f(x_i) \Bigr|, \\
I_2 &\defeq
\Bigl| \frac{1}{\E \tilde w_K^1}
  \sum_{i=1}^N \frac{1}{N} \tilde w_K^i f(x_i)
  - \ip{f,\pi_{K+1}} \Bigr|.
\end{align}

We observe that
\begin{align}
\label{e:I1-bound}
I_1
&= \Bigl| 1 - \frac{R}{\E \tilde w_K^1} \Bigr|
   \Bigl| \sum_{i=1}^N w_K^i f(x_i) \Bigr| \\
&\le
\Bigl| 1 - \frac{R}{\E \tilde w_K^1} \Bigr|
\norm{f}_\infty
= \Bigl| \frac{1}{N} \sum_{i=1}^N
\Bigl( \frac{\tilde w_K^i}{\E \tilde w_K^i} - 1 \Bigr) \Bigr|
\norm{f}_\infty ,
\\
\label{e:I2-bound}
I_2
&=
\Bigl|
\frac{1}{N} \sum_{i=1}^N
\Bigl(
\frac{\tilde w_K^i f(x_i)}{\E \tilde w_K^i}
- \ip{f,\pi_{K+1}}
\Bigr)
\Bigr|.
\end{align}

We recall that for any i.i.d.\ random variables $(Y_i)_{i=1}^N$ in $L^2(\P)$,
\begin{equation}\label{e:y-iid-sum-bias-variance}
\E\brak*{\Bigl( \frac{1}{N} \sum_{i=1}^N Y_i \Bigr)^2}
= \frac{1}{N} \var(Y_1)
+ \bigl( \E Y_1 \bigr)^2.
\end{equation}

Applying~\eqref{e:break-to-I1-I2} with $x_i = X_K^i$ and combining~\eqref{e:I1-bound}, \eqref{e:I2-bound}, and~\eqref{e:y-iid-sum-bias-variance}, we obtain
\begin{equation}\label{e:I3-I4-sum-bound}
\Biggl\|
\sum_{i=1}^N w_K^i f(X_K^i)
- \ip{f,\pi_{K+1}}
\Biggr\|_{L^2(\P)}
\le I_3 + I_4,
\end{equation}
where
\begin{align}
\label{e:I3-bound}
I_3
  &=
  \frac{\norm{f}_{L^\infty}}{\sqrt{N}}
\var\!\Bigl(
\frac{\tilde w_K^i}{\E \tilde w_K^i}
\Bigr)^{1/2}
\overset{\eqref{e:variance-estimate}}{\le}
\norm{f}_\infty
\sqrt{\frac{4C_w}{N}}
  \,,
\\
\label{e:I4-bound}
I_4
&=
\frac{1}{\sqrt{N}}
\var\!\Bigl(
\frac{\tilde w_K^i f(X_K^i)}{\E \tilde w_K^i}
\Bigr)^{1/2}
+
\Bigl|
\E \frac{\tilde w_K^i f(X_K^i)}{\E \tilde w_K^i}
- \ip{f,\pi_{K+1}}
\Bigr| \\
&\overset{\eqref{e:variance-estimate}, \eqref{e:normalized-bias}}{\le}
\sqrt{\frac{4C_w}{N}} \norm{f}_\infty
+ 2^{\,2 - T / t_{\mix,1}^\infty}
\norm{f}_\infty .
\end{align}

Finally, choosing $N$ and $T$ as in~\eqref{e:NT-choice} and combining
\eqref{e:I3-I4-sum-bound}, \eqref{e:I3-bound}, and~\eqref{e:I4-bound}
proves~\eqref{e:ais-small-error-and-wki-def}, as desired.
\smallskip

We now prove~\eqref{e:empHs}.
The bound~\eqref{e:empHs} follows quickly from~\eqref{e:ais-small-error-and-wki-def}, and the Fourier representation of the~$\dot H^{-s}$ norm.
Recall, if~$\varphi$ is a distribution in the homogeneous Sobolev space~$H^{-s} = H^{-s}(\T^d)$, then the norm is equivalently defined by
\begin{equation}\label{e:HsNormDef}
  \norm{\varphi}_{\dot H^{-s}}^2
    \defeq \sum_{n \in \Z^d - \set{0}} \frac{\abs{\ip{\varphi, e_n}}^2}{\abs{n}^{2s}}
    \quad\text{where}\quad
    e_n(x) \defeq e^{2 \pi i n \cdot x}
    \,.
\end{equation}

Now,
we note that for any~$n \in \Z^d - \set{0}$, the bound~\eqref{e:ais-small-error-and-wki-def} implies
\begin{equation}
  \E \ip{\mu_{K+1, N} - \pi_{K+1}, e_n}^2 \leq \delta^2\,.
\end{equation}
Thus
\begin{align}
  \E \norm{\mu_{K+1, N} - \pi_{K+1}}^2
    &= \sum_{n \in \Z^d - \set{0}}
	\E \frac{\ip{\mu_{K+1, N} - \pi_{K+1}, e_n}^2}{\abs{n}^{2s}}
    \leq \sum_{n \in \Z^d - \set{0}} \frac{\delta^2}{\abs{n}^{2s}}
  \\
    &\leq C_s \delta^2\,,
\end{align}
where
\begin{equation}\label{e:CsDef}
  C_s \defeq \sum_{n \in \Z^d - \set{0}} \frac{1}{\abs{n}^{2s}}
  \,.
\end{equation}
This concludes the proof.
\end{proof}

\subsection{The effective sample size (Proposition~\ref{p:essBound})}
We now prove Proposition~\ref{p:essBound} bounding the effective sample size.
The bound follows from a more general fact about normalized sums of nonnegative i.i.d.\ random variables.

\begin{proof}[Proof of Proposition~\ref{p:essBound}]
  Let~$\set*{\tilde \zeta_i}_{1\leq i \leq N}$ be non-negative i.i.d.\ $L^2$ random variables with~$\tilde \mu = \E\brak*{\tilde \zeta_1}$ and~$\tilde \sigma^2 = \var\paren*{\tilde \zeta_1}$ and define
  \begin{equation}
    \zeta_i = \frac{\tilde \zeta_i }{ S_N}
    \quad\text{where}\quad
    S_N = \sum_{i=1}^N \tilde \zeta_i
    \,.
  \end{equation}
  We claim that we must have
\begin{equation}\label{e:ess-ais-langevin}
\E\brak*{\sum_{i=1}^N \zeta_i^2} \leq \frac{8\paren*{\tilde \sigma^2 + \tilde \mu^2}}{N\tilde \mu^2}\,.
\end{equation}

  Momentarily postponing the proof of~\eqref{e:ess-ais-langevin}, we now prove~\eqref{e:ais-ssw-ub} and~\eqref{e:ais-ess-lb}.
  Setting~$\tilde \zeta_i = \tilde w_K^i$ and taking~$f \equiv 1$ in~\eqref{e:variance-estimate} gives
\begin{equation}
\tilde \sigma^2 \leq 4 C_w \tilde \mu^2\,.
\end{equation}
  Combining this with~\eqref{e:ess-ais-langevin} yields~\eqref{e:ais-ssw-ub} as desired.
Using this and Jensen's inequality, implies the lower bound on the effective sample size in~\eqref{e:ais-ess-lb}.
\medskip

It remains to prove the claim~\eqref{e:ess-ais-langevin}.
  Let~$\bar S_N = S_N / N$. Then
\begin{align}
  \E\brak*{\sum_{i=1}^N \zeta_i^2}
    &= N\E \brak*{\frac{\tilde \zeta_1^2}{S_N^2}\one_{\set{\bar S_N \geq \frac{\tilde \mu}{2}}}}
      + \E\brak*{\frac{1}{S_N^2}\sum_{i=1}^N \tilde \zeta_i^2 \one_{\set{\bar S_N < \frac{\tilde \mu}{2}}}}\\
   &\leq \frac{4\paren*{\tilde \sigma^2 + \tilde \mu^2}}{N\tilde \mu^2}
      + \P\brak*{\abs*{\bar S_N - \tilde \mu} > \frac{\tilde \mu}{2}}\\   
   &\leq \frac{8\paren*{\tilde \sigma^2 + \tilde \mu^2}}{N\tilde \mu^2}\,.
\end{align}
Here we used~$\sum_{i=1}^N \zeta_i^2 \leq 1$ to obtain the first inequality, and then Chebyshev's inequality to obtain the second inequality.
This concludes the proof.
\end{proof}

\subsection{The \texorpdfstring{$\chi^2$}{chi2} divergence (Propositions \ref{p:aisChi2} and \ref{p:aisGaussian})}\label{s:chi2Proof}

The proof of Proposition~\ref{p:aisGaussian} is short and direct.

\begin{proof}[Proof of Proposition~\ref{p:aisGaussian}]
  Notice $P_k^\infty r_k^2 = \ip{r_k^2, \pi_k}$.
  Thus
  \begin{align}
    \prod_{k = 1}^K \norm{P_{k}^\infty r_{k}^2}_{L^\infty}
      & = 
      \prod_{k = 1}^K \ip{ r_k^2, \pi_k}
	 = \exp\paren[\bigg]{
	    \sum_{k = 1}^K
	      \ln \ip{r_k^2 , \pi_k}
	  }
    \\
      \label{e:prodR2}
	&\leq
	  \exp\paren[\bigg]{
	    \sum_{k = 1}^K
	      \paren[\big]{ \ip{r_k^2,\pi_k} - 1 }
	  }
	= \exp\paren[\bigg]{
	    \sum_{k = 1}^K \chi^2(\pi_{k+1}; \pi_{k})
	  }
      \,,
  \end{align} 
  which proves~\eqref{e:prodPkInfChi2}.
\end{proof}

The proof of Proposition~\ref{p:aisGaussian} is also a direct calculation.
\begin{proof}[Proof of Proposition~\ref{p:aisGaussian}]
  For the second assertion, the choice of temperatures described in the statement reduces to choosing
  \begin{equation}\label{e:epsilonk-def}
    \epsilon_k = \frac{\epsilon K}{(k - 1)(1 - \epsilon) + \epsilon K}
    \,.
  \end{equation}
  Since~$\pi_k$ is the Gaussian~$\mathcal N(0, \epsilon_k I_d)$, the~$\chi^2$ divergence can be computed explicitly.
  Using, for instance Proposition~17.1 in~\cite{ChopinPapaspiliopoulos20}), we see
  \begin{equation}\label{e:chi2Gaussian}
    \chi^2( \pi_{k+1}, \pi_k)
    = \paren[\bigg]{ \frac{\epsilon_k^2}{2 \epsilon_k \epsilon_{k+1} - \epsilon_{k+1}^2} }^{d/2} - 1
    \,.
  \end{equation}

  To obtain~\eqref{e:prodPkInfGaussian}, set
  \begin{equation}
    \delta_k \defeq 1 - \frac{\epsilon_{k+1}}{\epsilon_k}
      = \paren[\Big]{\frac{1 - \epsilon}{\epsilon K}} \epsilon_{k+1}
      \,,
  \end{equation}
  and use~\eqref{e:chi2Gaussian} to obtain
  \begin{equation}\label{e:chi2}
    \chi^2(\pi_{k+1}; \pi_k)
      = \paren[\bigg]{
	1 + \frac{\delta_k^2}{1 - \delta_k^2}
      }^{d/2} - 1
      \leq C d \delta_{k+1}^2
      \leq \frac{C d \epsilon_{k+1}^2}{\epsilon^2 K^2}
      \,.
  \end{equation}
  Hence using~\eqref{e:prodR2}, \eqref{e:epsilonk-def} and~\eqref{e:chi2} we obtain
  \begin{align}
    \prod_{k = 1}^K \norm{P^\infty_k r_k^2}
      &\leq
	\exp\paren[\bigg]{ \sum_{1}^K \frac{C d}{\paren[\big]{(k - 1)(1 - \epsilon) + \epsilon K}^2 } }
      \leq
	\exp\paren[\Big]{
	  \frac{C d}{\epsilon K - 1}
	}
      \leq C
      \,,
  \end{align}
  proving~\eqref{e:prodPkInfGaussian} as desired.
\end{proof}


\section{Proof of convergence for Langevin AIS}\label{s:langevin}

We now prove Theorem~\ref{t:langevinGen}.
We assume, without loss of generality, that the first initial temperature~$\epsilon_1 = 1$.

\subsection{Proofs of Proposition \ref{p:langevinGen} and Theorem \ref{t:langevinGen}}
We prove Proposition~\ref{p:langevinGen} and Theorem~\ref{t:langevinGen} by verifying that the transition kernels for~\eqref{e:langevin} satisfy Assumptions~\ref{a:reversible}--\ref{a:L1Linf}, and then use Proposition~\ref{p:expectedDiff} and Theorem~\ref{t:ais}.
The fact that Assumption~\ref{a:reversible} holds is well known.
Checking Assumptions~\ref{a:r2} and~\ref{a:L1Linf} require a little work and we state them as the following lemmas.

\begin{lemma}\label{l:uMixLangevin}
  Let~$P_{\epsilon, t}$ be the transition kernel of~\eqref{e:langevin} at time~$t$.
  If~$U$ satisfies Assumptions~\ref{a:criticalpts}--\ref{assumption: nondegeneracy}, then the uniform mixing time of~$P_{1}$ is
  finite (i.e.\ for $\epsilon_1 = 1$, $P_{\epsilon_1}$ satisfies Assumption~\ref{a:L1Linf}).
\end{lemma}

\begin{lemma}\label{l:prodBoundLangevin}
  Let $U$ be a double well potential that satisfies Assumption~\ref{a:criticalpts}--\ref{a:massRatioBound} with $\epsilon_{\min} < 1 \leq \epsilon_{\max}$.
  There exists constants~$\bar C_w = \bar C_w(U, \nu)$ and $C_T = C_T(U)$ such that the following holds.
For any~$\epsilon \in (\epsilon_{\min}, 1]$, choose~$K = \ceil*{1/(\epsilon\nu)}$ and choose~$\epsilon_2$, \dots, $\epsilon_{K+1} = \epsilon$ so that~$\set{1/\epsilon_k}_{1 \leq k \leq K+1}$ are linearly spaced. Then, Assumption~\ref{a:r2} holds with
\begin{equation}\label{e:CwT0}
C_w = \bar C_w \quad\text{and}\quad T_0=C_T\paren*{\frac{1}{\epsilon} + \log K}\,.
\end{equation}  
\end{lemma}

Given Lemmas~\ref{l:uMixLangevin} and~\ref{l:prodBoundLangevin}, the proofs of Proposition~\ref{p:langevinGen} and Theorem~\ref{t:langevinGen} are immediate.

\begin{proof}[Proof of Proposition~\ref{p:langevinGen}]
  It is well known (see for instance Chapter~8 in~\cite{Kolokoltsov00}, or Chapter 4.6 in~\cite{Pavliotis14}) that~$P_{\epsilon, \cdot}$, the transition kernel of~\eqref{e:langevin}, is reversible, and~$\pi_\epsilon$ is the unique stationary distribution.
  The choice of $K$ in~\eqref{e:KTchoice}, and by Lemmas~\ref{l:uMixLangevin}--\ref{l:prodBoundLangevin}, will now guarantee that Assumptions~\ref{a:reversible}--\ref{a:L1Linf} are satisfied, with the constants $C_w$ and $T_0$ given in~\eqref{e:CwT0}.
Therefore, Proposition~\ref{p:expectedDiff} applies and yields the bias and variance estimates~\eqref{e:normalized-bias-langevin}--\eqref{e:variance-estimate-langevin}.  
\end{proof}

\begin{proof}[Proof of Theorem~\ref{t:langevinGen}]
  As with the proof of Proposition~\ref{p:langevinGen} presented above,
  Assumptions~\ref{a:reversible}--\ref{a:L1Linf} hold, with $C_w$ and $T_0$ as specified in~\eqref{e:CwT0}.  
Consequently, Theorem~\ref{t:ais} applies and yields~\eqref{e:emp-error-bd}--\eqref{e:emp-HsBd}.  
\end{proof}

It remains to prove Lemmas~\ref{l:uMixLangevin} and~\ref{l:prodBoundLangevin}.
Their proofs require certain spectral estimates, which are described in Section~\ref{s:spectral}.
Following this, we prove Lemmas~\ref{l:uMixLangevin}, \ref{l:prodBoundLangevin} in Section~\ref{s:LangevinLemmaProofs}.

\subsection{Spectral decomposition}\label{s:spectral}

Let~$\mathcal L_\epsilon$, defined by
\begin{equation}
  \mathcal L_\epsilon f = -\grad U \cdot \grad f + \epsilon \lap f
  \,,
\end{equation}
be the generator of~\eqref{e:langevin}.
We know that for any test function~$f$, the action of~$P_{\epsilon, t}$ on~$f$ satisfies the \emph{Kolmogorov backward equation} (see for instance~\cite[Chapter 2]{Pavliotis14}).
That is, if
\begin{equation}
  u_t(x) = P_{\epsilon, t} f(x) \defeq \int_{\T^d} P_{\epsilon, t}(x, dy) f(y)
  \,,
\end{equation}
then
\begin{equation}\label{e:KBackward}
  \partial_t u = \mathcal L_\epsilon u
  \quad\text{and}\quad
  u_0 = f\,.
\end{equation}

We now state certain spectral properties of~\eqref{e:langevin} which will be used in the proofs of Lemmas~\ref{l:uMixLangevin} and~\ref{l:prodBoundLangevin}.
Let~$L^2(\pi_\epsilon)$ denote the weighted~$L^2$ space with inner-product and norm defined by
\begin{equation}
  \ip{f, g}_{L^2(\pi_\epsilon)}
    \defeq \int_{\T^d} f g \, \pi_\epsilon \, dx
  \quad\text{and}\quad
  \norm{f}^2_{L^2(\pi_\epsilon)} \defeq 
    \int_{\T^d} \abs{f}^2 \pi_\epsilon \, dx
    \,,
\end{equation}
respectively.
It is well known (see for instance~\cite[Chapter 8]{Kolokoltsov00}) that the operator~$-\mathcal L_\epsilon$ is a self-adjoint operator the weighted space~$L^2(\pi_\epsilon)$, and has a discrete spectrum with positive eigenvalues.
We denote the eigenvalues by
\begin{equation}\label{e:evalOrder}
  0 = \lambda_{1, \epsilon}
    < \lambda_{2, \epsilon}
    \leq \lambda_{3, \epsilon}
    \cdots
    \,,
\end{equation}
and the corresponding~$L^2(\pi_\epsilon)$ normalized eigenfunctions by~$\psi_{1, \epsilon}$, $\psi_{2, \epsilon}$, etc.
We note that the smallest eigenvalue~$\lambda_{1, \epsilon}$ is~$0$, and the second smallest eigenvalue~$\lambda_{2, \epsilon}$ is strictly positive (i.e.\ $P_\epsilon$ has a spectral gap).
In fact, for the proof of Lemmas~\ref{l:uMixLangevin} and~\ref{l:prodBoundLangevin} the properties we need were collected in a convenient form in~\cite{HanIyerEA26} as Properties~4.6--4.8.
We quote the portions we need here for easy reference.

\begin{property}[Eigenvalue bounds]\label{p:spectral}
  For every~$H > \hat U$, there exists a constant~$C_H$ such that for every~$\epsilon \in (0, 1]$ we have
\begin{equation}\label{e:egvalLEpsilon}
  \lambda_{2,\epsilon}
    \geq C_H \exp\paren[\Big]{-\frac{H}{\epsilon}}
  .
\end{equation}
  Also, there exists~$\Lambda$ such that for all~$\epsilon \in (0, 1]$ such that
  \begin{equation}\label{e:lambdaiLower}
    \lambda_{i,\epsilon}\geq \Lambda,\quad 
    \text{for all } i\geq 3
    .
  \end{equation}
\end{property}

\begin{property}[Eigenfunction variation]\label{p:var}
  The function~$\epsilon \mapsto \pi_\epsilon(\Omega_1)$ is of bounded variation on the interval~$(0, 1]$.
  Moreover, for every~$\gamma < \hat \gamma$, there exists a constant~$C_\gamma$ such that for every~$0 < \epsilon' < \epsilon \leq 1$ we have
  \begin{gather}
    \label{e:PsiEpPi}
    \abs[\big]{\ip{\psi_{2, \epsilon}, \pi_{\epsilon'}}}
      \leq 
    C_\gamma\paren[\Big]{\exp\paren[\Big]{-\frac{\gamma}{\epsilon}}+\abs[\big]{\pi_{\epsilon'}(\Omega_1)-\pi_{\epsilon}(\Omega_1)}}
    \,.
  \end{gather}
\end{property}

\begin{property}[Eigenfunction bounds]\label{property:eigenfunctions}
  There exists constant~$C_{\psi_2}$, independent of~$\epsilon$ such that 
  \begin{equation}\label{eq:defCpsi}
    \sup\limits_{0< \epsilon\leq 1}\|\psi_{2,\epsilon}\|_{L^{\infty}(\mathbb{T}^d)}\leq C_{\psi_2}.
  \end{equation}
\end{property}

Under Assumptions~\ref{a:criticalpts}--\ref{a:criticalpts}, Section 9 in~\cite{HanIyerEA26} shows that Properties~\ref{p:spectral}--\ref{property:eigenfunctions} hold.
To briefly discuss the significance of Properties~\ref{p:spectral}--\ref{property:eigenfunctions}, we note that convergence of (reversible) Markov processes can effectively be studied using the spectral decomposition (see for instance Chapter 12 in~\cite{LevinPeres17}).
In particular, for reversible Markov processes the rate of convergence is controlled both above and below by the \emph{spectral gap}, which in our case is simply~$\lambda_{2, \epsilon}$.

For the Langevin system with a multimodal potential, the lower bound~\eqref{e:egvalLEpsilon} is sharp (see~\cite[Chapter 8, Proposition 2.2]{Kolokoltsov00}, or~\cite{MenzSchlichting14}), and so spectral gap is exponentially small.
Thus sampling by simulating~\eqref{e:langevin} directly, will necessarily cost~$O(e^{\hat U/\epsilon})$, which is not desirable.

However, in our situation the third eigenvalue is large (i.e.\ bounded independent of~$\epsilon$), as asserted by~\eqref{e:lambdaiLower} in Property~\ref{p:spectral}.
This will give fast convergence \emph{provided} we control the projection onto the second eigenspace.
Variants of this idea have been used by several authors in many contexts to accelerate convergence~\cite{
  ConstantinKiselevEA08,
  KwokLauEA13,
  FengIyer19
}.
A warm start to Langevin dynamics will also control the projection on the second eigenspace, and this was recently used by Koehler, Lee, and Vuong~\cite{KoehlerLeeEA25} in a related multimodal sampling problem.


In~\cite{HanIyerEA26}, the authors related the projection onto the second eigenspace as a \emph{mass imbalance}, and controlled it by resampling points with weights proportional to the ratio of the densities.
This is a standard technique used in sequential Monte Carlo algorithms (see for instance~\cite{DoucetFreitasEA01,Liu08,ChopinPapaspiliopoulos20}), and is applicable in many situations of practical interest.
The disadvantage of this approach, however, is that one loses independence of realizations, and as a result the estimates are harder to obtain and weaker than in the i.i.d.\ case.

In our situation, we reweight points instead of resampling, and the variance is controlled by the product~\eqref{e:Ptrk2-ub} appearing in the constant~$C_w$.
That is, the reweighting controls the error in terms of how fast the kernel~$P_{\epsilon_k}$ mixes~$r_{\epsilon_k}^2$.
This in turn is controlled by the projection onto the second eigenspace, which we will bound using Properties~\ref{p:var} and~\ref{property:eigenfunctions} (see Lemma~\ref{l:eTLk-tilde-rk^2}, below for the precise estimates).

We now use Properties~\ref{p:spectral}--\ref{property:eigenfunctions} to prove Lemmas~\ref{l:uMixLangevin} and~\ref{l:prodBoundLangevin}.
We first control the high order terms in the spectral decomposition, and then control the terms in the product~\eqref{e:Ptrk2-ub}.
These steps are stated as the next two lemmas.

\begin{lemma}\label{l:psi-Linfty-weighted-sum}
    Let $\epsilon > 0$. Suppose there exist constants $\Lambda_1 > 0$ and $m \in \N$ such that $\lambda_{m,\epsilon} \ge \Lambda_1$. Then there exists a constant 
    $\tilde C_{\psi} = \tilde C_{\psi}\bigl(\norm{U}_{C^{\ceil{d/2} \vee 2}}, d, \Lambda_1\bigr)$, 
    independent of $\epsilon$, such that for
    $T_0(\Lambda_1) \defeq (d+2)/\Lambda_1$, 
    and any~$T \geq T_0$,
    \begin{equation}\label{e:psi-Linfty-weighted-sum}
        \sum_{i=m}^\infty e^{-\lambda_{i, \epsilon} T}\norm*{\psi_{i, \epsilon}}_\infty^2 \leq \tilde C_{\psi}Z_\epsilon \min\set{\epsilon, 1}^{-(d+1)} \exp\paren*{\frac{\norm*{U}_\infty}{\epsilon}-\Lambda_1 T}\,.
    \end{equation}
\end{lemma}

\begin{lemma}\label{l:eTLk-tilde-rk^2}
  Let~$\tilde C_\psi$ be the constant from Lemma~\ref{l:psi-Linfty-weighted-sum} with~$\Lambda_1 = \Lambda$ from Property~\ref{p:spectral}.
Let~$C_\gamma, C_{\psi_2}$ be the constants defined as in Property~\ref{p:var} and~\ref{property:eigenfunctions}, respectively. Then, there exists a constant~$C_Z(d,U)$ such that 
if
\begin{equation}\label{e:Tchoice}
  T \geq \frac{1}{\Lambda} \max\set*{\frac{\norm{U}_\infty}{\epsilon} + \log K + \paren*{\frac{d}{2} + 1} \abs{\ln \epsilon} + \log\paren*{\tilde C_\psi C_Z}, d+2  }\,,
\end{equation}
then for each level $k\in \set*{1, \ldots, K}$,
\begin{align}
  \norm*{P_{\epsilon_k, T} \tilde r_k^2}_{L^\infty} \leq \frac{Z_{\epsilon_{k+2}}}{Z_{\epsilon_k}} \biggl(
    1 &+ \frac{1}{K}
  \\
  \label{e:eTLk-tilde-rk^2-Linfty-ub}
    &+ C_{\psi_2}C_\gamma \paren[\Big]{\exp\paren[\Big]{-\frac{\hat \gamma}{2\epsilon_k}} + \abs*{\pi_{\epsilon_{k+1}}\paren*{\Omega_1}-\pi_{\epsilon_k}\paren*{\Omega_1}}} \biggr)
    \,.
\end{align}
\end{lemma}

Lemmas~\ref{l:psi-Linfty-weighted-sum} and~\ref{l:eTLk-tilde-rk^2} are proved in Sections~\ref{s:psi-Linfty-weighted-sum} and~\ref{s:eTLk-tilde-rk^2} respectively.

\subsection{Proofs of Lemmas~\ref{l:uMixLangevin} and~\ref{l:prodBoundLangevin}}\label{s:LangevinLemmaProofs}


We now use Lemmas~\ref{l:psi-Linfty-weighted-sum} and~\ref{l:eTLk-tilde-rk^2} to prove Lemmas~\ref{l:uMixLangevin} and~\ref{l:prodBoundLangevin}.
For the remainder of this section, we slightly abuse notation and write
$Z_k, \pi_k, P_{k,t}, \psi_{i,k}, \lambda_{i,k}$ in place of
$Z_{\epsilon_k}, \pi_{\epsilon_k}, P_{\epsilon_k,t}, \psi_{i,\epsilon_k}, \lambda_{i,\epsilon_k}$.

\begin{proof}[Proof of Lemma~\ref{l:uMixLangevin}]
It suffices to prove that there exists $T_0 > 0$ such that for any~$f\in L^1(\pi_{1})$, 
\begin{equation}\label{e:L1-to-Linfty-at-first-level}
\norm{P_{1, T_0} f - \ip{f, \pi_{1}}}_{L^\infty} \leq \frac{1}{2}\norm{f}_{L^1(\pi_1)}\,.
\end{equation}
Indeed, set $\epsilon = \epsilon_1 = 1$, $\Lambda_1 = \lambda_{2,1}$, and $m = 2$ in Lemma~\ref{l:psi-Linfty-weighted-sum}. 
It follows that there exists a constant $C_{\epsilon_1}>0$ such that for any $T \ge (d+2)/\lambda_{2,1}$,
\begin{equation}\label{e:level-one-psi-infty-ws}
    \sum_{i=2}^\infty e^{-\lambda_{i,1} T}\norm*{\psi_{i,1}}_\infty^2 
    \le C_{\epsilon_1} \exp\paren*{-\lambda_{2,1} T}\,.
\end{equation}

Then, using H\"older's inequality yields
\begin{align}
\norm{P_{1, T} f - \ip{f, \pi_{1}}}_{L^\infty} &= \norm*{\sum_{i=2}^\infty \exp\paren*{-\lambda_{i, 1} T} \ip{f, \psi_{i, 1}}_{L^2(\pi_1)} \psi_{i,1}}_{L^\infty} \\
&\leq \sum_{i=2}^\infty \exp\paren*{-\lambda_{i,1} T}\norm*{\psi_{i,1}}_{L^\infty}^2  \norm{f}_{L^1(\pi_1)}\\
&\leq C_{\epsilon_1} \exp\paren*{-\lambda_{2,1} T} \norm{f}_{L^1(\pi_1)}\,.
\end{align}
Hence, choosing~$T_0 = \max\set{d+2, \log \paren{2C_{\epsilon_1}}  } /\lambda_{2,1}$ implies~\eqref{e:L1-to-Linfty-at-first-level} and concludes the proof.
\end{proof}

\begin{proof}[Proof of Lemma~\ref{l:prodBoundLangevin}]
Choosing
\begin{equation}\label{e:CTdef}
  C_T = \frac{1}{\Lambda}\max\set*{\norm{U}_\infty  +  \paren[\Big]{\frac{d}{2} + 1} + \log (\tilde C_\psi C_Z), d + 2}\,,
\end{equation}
  and~$T_0$ as in~\eqref{e:CwT0} implies that if~$T \geq T_0$  then~\eqref{e:Tchoice} holds.
  Thus we may apply Lemma~\ref{l:eTLk-tilde-rk^2} to obtain
\begin{equation}\label{e:prod-eTLk-tilde-rk^2-ub}
    \prod_{k=1}^{K}\norm*{P_{k,T} \tilde r_{k}^2}_{L^\infty}
    \leq \frac{Z_{K+1} Z_{K+2}}{Z_1 Z_2}
        \prod_{k=1}^{K}\Theta(k, k+1)\,,
\end{equation}
where
\begin{equation}\label{e:Theta-def}
    \Theta(k,k+1)\defeq
    \paren*{1 + C_{\psi_2}C_\gamma\paren*{
        \exp\paren*{-\frac{\hat \gamma}{2\epsilon_k}}
        + \abs*{\pi_{k+1}\paren*{\Omega_1}-\pi_k\paren*{\Omega_1}}}
        +\frac{1}{K}}\,.
\end{equation}
By the AM-GM inequality this implies
\begin{equation}\label{e:amgm}
  \prod_{k=1}^{K}\norm*{P_{k,T} \tilde r_{k}^2}_{L^\infty}
    \leq \frac{Z_{K+1} Z_{K+2}}{Z_1 Z_2}
    \paren[\Big]{ \frac{1}{K} \sum_{k = 1}^K \Theta(k, k+1) }^K
  \,.
\end{equation}

  Since~$\set{1/\epsilon_k}$ are linearly spaced between~$1$ and~$1/\epsilon$, they are given by~\eqref{e:epsilonk-def},
  and hence
  \begin{equation}\label{e:cgeom}
    \sum_{k=1}^{K}
        \exp\paren*{-\frac{\hat \gamma}{2\epsilon_k}}
      \leq \frac{e^{-\hat \gamma / 2}}{1 - \exp\paren[\big]{\frac{-\hat \gamma}{2}\paren{\nu - 1}}}
      \defeq C_{\mathrm{geom}} < \infty
      \,,
  \end{equation}
  where we note that~$C_{\mathrm{geom}}$ is an explicit constant that is independent of~$\epsilon$.

Moreover,
\begin{equation}\label{e:tmppk1pk1}
  \sum_{k=1}^{K}
      \abs*{\pi_{k+1}\paren*{\Omega_1}-\pi_k\paren*{\Omega_1}}
    \leq \int_\epsilon^{1}
        \abs*{\partial_{\epsilon'}\pi_{\epsilon'}(\Omega_1)} \, d\epsilon'
  \,.
\end{equation}
Assumptions~\ref{a:criticalpts}--\ref{a:massRatioBound} and Lemma~8.2 in~\cite{HanIyerEA26} imply the existence of a constant $C_\mathrm{BV}>0$ such that
\begin{equation}\label{e:tmppk1pk2}
    \int_\epsilon^{1}
        \abs*{\partial_{\epsilon'}\pi_{\epsilon'}(\Omega_1)} \, d\epsilon'
    \leq C_\mathrm{BV}\,.
\end{equation}

Combining~\eqref{e:cgeom}, \eqref{e:tmppk1pk1} and~\eqref{e:tmppk1pk2} we obtain
\begin{equation}\label{e:prod-Thetak-ub}
    \paren[\bigg]{ \frac{1}{K} \sum_{k = 1}^K \Theta(k, k+1) }^K
    \leq \paren*{1
        + \frac{C_{\psi_2}C_\gamma(C_\mathrm{geom} + C_\mathrm{BV})+1}{K}}^{K}
    \leq C_\Theta\,,
\end{equation}
where
\begin{equation}
C_\theta \defeq \exp\paren*{C_{\psi_2}C_\gamma(C_\mathrm{geom} + C_\mathrm{BV})+1}\,.
\end{equation}
Combining~\eqref{e:prod-eTLk-tilde-rk^2-ub} and~\eqref{e:prod-Thetak-ub},
together with the fact that 
\begin{equation}
\prod_{k=1}^{K}\norm*{P_{k,T} r_{k}^2}_{L^\infty} = \frac{Z_1^2}{Z_{K+1}^2} \prod_{k=1}^{K}\norm*{P_{k,T} \tilde r_{k}^2}_{L^\infty}\,,
\end{equation}
and $Z_{K+2} \leq Z_{K+1}$,
proves~\eqref{e:Ptrk2-ub} with
\begin{equation}\label{e:T0-tildeCw-choice}
C_w = 2C_\Theta\,.
\qedhere
\end{equation}
\end{proof}

\subsection{Proof of Lemma \ref{l:eTLk-tilde-rk^2}}\label{s:eTLk-tilde-rk^2}
To complete the proofs of Proposition~\ref{p:langevinGen} and Theorem~\ref{t:langevinGen}, it only remains to prove Lemmas~\ref{l:psi-Linfty-weighted-sum} and~\ref{l:eTLk-tilde-rk^2}.
Since the proof of Lemma~\ref{l:eTLk-tilde-rk^2} is shorter, we present it first.
\begin{proof}[Proof of Lemma~\ref{l:eTLk-tilde-rk^2}]
Using~\eqref{e:KBackward} and the spectral decomposition of~$\mathcal L_\epsilon$, we have
\begin{align}
    P_{k, T} \tilde r_k^2 
    &= \ip{\tilde r_k^2 , \pi_k} 
    + e^{-\lambda_{2,k} T} \ip{\tilde r_k^2, \psi_{2,k}}_{L^2(\pi_k)} \psi_{2,k} 
    + \sum_{i=3}^\infty e^{-\lambda_{i,k} T}
        \ip{\tilde r_k^2, \psi_{i,k}}_{L^2(\pi_k)}\, \psi_{i,k}\,.
\end{align}
Notice
  \begin{equation}\label{e:rk2pik}
  \ip{\tilde r_k^2, \tilde \pi_k}
    = \int_{\T^d} \exp\paren[\bigg]{
	-U \paren[\Big]{
	  \frac{2}{\epsilon_{k+1}}-\frac{1}{\epsilon_k}}
      } \, dx
    \overset{\eqref{e:epsilonk-def}}{=} \int_{\T^d} \exp\paren[\bigg]{
	\frac{-U}{\epsilon_{k+2}}
      }
      \, dx
    = Z_{k+2}\,.
\end{equation}
Similarly,
\begin{align}
  \label{e:rk2pik2}
\ip{\tilde r_k^2, \psi_{2,k}}_{L^2(\pi_k)} &= \frac{Z_{k+2}}{Z_k} \ip{\psi_{2,k}, \pi_{k+2}}\,,
  \\
  \label{e:rk2pik3}
    \abs*{\ip{\tilde r_k^2, \psi_{i,k}}_{L^2(\pi_k)}} &\leq \norm{\tilde r_k^2}_{L^1(\pi_k)} \norm{\psi_{i,k}}_{L^\infty}\,,
\end{align}
and hence
\begin{equation}\label{e:eTLk-tilde-rk^2-sd-Linfty-ub}
    \norm{P_{k, T} \tilde r_k^2}_{L^\infty }
    \le \frac{Z_{k+2}}{Z_k} \paren*{
        1 
        + \abs*{\ip{\psi_{2,k}, \pi_{k+2}}} \norm{\psi_{2,k}}_{L^\infty }
        + \sum_{i=3}^\infty e^{-\lambda_{i,k} T}\norm{\psi_{i,k}}_{L^\infty}^2
    }\,.
\end{equation}
We now estimate the second and third terms of the right hand side separately.
For the second term, using~\eqref{eq:defCpsi} and~\eqref{e:PsiEpPi} for~$\gamma=\hat \gamma / 2$ yields
\begin{equation}\label{e:second-term-estimate}
\abs*{\ip{\psi_{2,k}, \pi_{k+2}}} \norm{\psi_{2,k}}_{L^\infty } \leq C_{\psi_2}C_\gamma\paren[\Big]{\exp\paren[\Big]{-\frac{\hat\gamma}{2\epsilon_k}}+\abs[\big]{\pi_{k+2}(\Omega_1)-\pi_{k}(\Omega_1)}}\,.
\end{equation}
To estimate the third term, 
we first apply Lemma~\ref{l:psi-Linfty-weighted-sum}. 
Set $\epsilon = \epsilon_k$, $\Lambda_1 = \Lambda$, and $m = 3$ in Lemma~\ref{l:psi-Linfty-weighted-sum}. 
It follows that there exists a constant $\tilde C_{\psi} > 0$, independent of $\epsilon$, such that for any
\begin{equation}\label{e:Tlarge-1}
    T \ge \frac{d+2}{\Lambda}\,,
\end{equation}
we have
\begin{equation}\label{e:psi-Linfty-ws-from-third-ev}
    \sum_{i=3}^\infty e^{-\lambda_{i,k} T}\norm*{\psi_{i,k}}_{L^\infty}^2 
    \le \tilde C_{\psi} Z_k \epsilon_k^{-(d+1)} 
        \exp\paren*{\frac{\norm*{U}_{\infty}}{\epsilon_k} - \Lambda T}\,.
\end{equation}

In particular, if $T$ further satisfies
\begin{equation}\label{e:Tlarge-2}
    T \ge \frac{1}{\Lambda}
	\paren[\bigg]{\frac{\norm{U}_\infty}{\epsilon_k} 
	+ \log\paren[\Big]{ \frac{\tilde C_{\psi} K Z_k}{\epsilon_k^{d+1}}} }\,,
\end{equation}
then combining this with~\eqref{e:psi-Linfty-ws-from-third-ev} yields
\begin{equation}\label{e:third-term estimate}
    \sum_{i=3}^\infty e^{-\lambda_{i,k} T}\norm*{\psi_{i,k}}_{L^\infty}^2 
    \le \frac{1}{K}\,.
\end{equation}

Finally, we estimate the normalizing constant~$Z_k$.
By the Laplace method (see for instance~\cite[Proposition B2]{Kolokoltsov00}),
\begin{equation}
  \lim_{\epsilon \to 0} \frac{1}{(2\pi \epsilon)^{d/2} \sqrt{\abs{\det \grad^2 U(x_{\min, 1})}} } \int_{\Omega_1} e^{-U / \epsilon} \, dx
  = 1\,.
\end{equation}
This implies that there exists a constant~$C_Z(d,U)$ such that
\begin{equation}
  Z_\epsilon \leq C_Z(d,U)\,\epsilon^{\frac{d}{2}}\,,
  \quad\text{for all } \epsilon \in (0, 1]\,.
\end{equation}
 It follows that the choice of~$T$ in~\eqref{e:Tchoice} ensures that~$T$ is sufficiently large to satisfy~\eqref{e:Tlarge-1} and~\eqref{e:Tlarge-2}, so that~\eqref{e:third-term estimate} holds. Combining~\eqref{e:eTLk-tilde-rk^2-sd-Linfty-ub}, \eqref{e:second-term-estimate}, and~\eqref{e:third-term estimate} then yields~\eqref{e:eTLk-tilde-rk^2-Linfty-ub}.
\end{proof}

\subsection{Proof of Lemma~\ref{l:psi-Linfty-weighted-sum}}\label{s:psi-Linfty-weighted-sum}

We now address Lemma~\ref{l:psi-Linfty-weighted-sum}.
We will bound the left hand side of~\eqref{e:fLinfty-ub-final} by obtaining a uniform in~$i$ bound on~$\norm{\psi_{i, \epsilon}}_{\infty}$, and then bound the sum using the asymptotic growth of the eigenvalues.
We state this as our next two lemmas.

\begin{lemma}\label{l:weylfunction-ub}
  Define the Weyl function by
  \begin{equation}
  N_{\mathcal L_\epsilon}(\lambda)
    = \abs{\set{i\in \N\st \lambda_{i, \epsilon} \leq \lambda}}
    = \sum_{\set{i\colon \lambda_{i, \epsilon} \leq \lambda}} 1
    \,.
  \end{equation}
  where~$\lambda_{i, \epsilon}$ are the eigenvalues of~$-\mathcal L_\epsilon$, ordered as in~\eqref{e:evalOrder}.
There exists a constant $C_W(\norm{U}_{C^2}, d)$ such that for any $\epsilon>0$ and $\lambda > 0$,
\begin{equation}\label{e:NLepsilon-ub}
N_{\mathcal L_\epsilon}(\lambda) \leq C_W\epsilon^{-\frac{d}{2}}(1+\lambda)^\frac{d}{2}\,.
\end{equation}
\end{lemma}

\begin{lemma}\label{l:eigenfunctions-Linfty-ub}
  For any $\Lambda_1 > 0$, there exists a constant $C_{\Lambda_1}(\norm{U}_{C^{1 + \floor{d/2}}}, d, \Lambda_1)>0$ such that for all $\epsilon > 0$ and any eigenvalue~$\lambda_{i, \epsilon} \geq \Lambda_1$, we have
\begin{equation}\label{e:fLinfty-ub-final}
  \norm{\psi_{i, \epsilon}}_{L^\infty(\T^d)}^2 \leq C_{\Lambda_1} Z_\epsilon \exp\paren[\Big]{\frac{\norm{U}_\infty}{\epsilon}}\paren[\Big]{\frac{\lambda_{i, \epsilon}}{\epsilon}}^{1 + \floor{\frac{d}{2}}}\,.
\end{equation}
\end{lemma}

Momentarily postponing the proofs of Lemmas~\ref{l:weylfunction-ub} and~\ref{l:eigenfunctions-Linfty-ub}, we now prove Lemma~\ref{l:psi-Linfty-weighted-sum}.
\begin{proof}[Proof of Lemma~\ref{l:psi-Linfty-weighted-sum}]
By Lemma~\ref{l:eigenfunctions-Linfty-ub}, for any $T \ge T_0$, we obtain
\begin{align}
  \MoveEqLeft
\sum_{i=m}^\infty e^{-\lambda_{i,\epsilon}T} \norm{\psi_{i,\epsilon}}_\infty^2 
\overset{\eqref{e:fLinfty-ub-final}}{\leq} 
C_{\Lambda_1} Z_\epsilon 
\exp\paren*{\frac{\norm{U}_\infty}{\epsilon}}\,
  \sum_{i=m}^\infty \paren[\Big]{\frac{\lambda_{i,\epsilon}}{\epsilon}}^{1 + \floor{\frac{d}{2}}} e^{-\lambda_{i,\epsilon}T}\\
  \label{e:sum-exp-lambda-psiLinfty}
&\leq 
C_{\Lambda_1} Z_\epsilon 
\exp\paren*{\frac{\norm{U}_\infty}{\epsilon} - \Lambda_1 (T-T_0)}\,
\sum_{i=m}^\infty \paren[\Big]{\frac{\lambda_{i,\epsilon}}{\epsilon}}^{1 + \floor{\frac{d}{2}}} e^{-\lambda_{i,\epsilon}T_0}\,.
\end{align}

Next, we estimate the remaining term 
$\sum_{i=m}^\infty \lambda_{i,\epsilon}^{1 + \floor{\frac{d}{2}}} e^{-\lambda_{i,\epsilon}T_0}$.
Define
\begin{equation}
    f_{T_0}(\lambda) \defeq \lambda^{1 + \floor{\frac{d}{2}}} e^{-\lambda T_0}, 
    \qquad
    g_{T_0}(\lambda) \defeq
    \begin{dcases}
        f_{T_0}\!\paren*{\tfrac{\Lambda_1}{2}}, & \lambda \in [0, \tfrac{\Lambda_1}{2}),\\[4pt]
        f_{T_0}(\lambda), & \lambda \in [\tfrac{\Lambda_1}{2}, \infty)\,.
    \end{dcases}
\end{equation}
Since $\lambda_{m,\epsilon} \ge \Lambda_1$ by assumption, we have
\begin{align}
\sum_{i=m}^\infty \lambda_{i,\epsilon}^{1 + \floor{\frac{d}{2}}} e^{-\lambda_{i,\epsilon}T_0}
  &= \sum_{i=m}^\infty g_{T_0}(\lambda_{i,\epsilon})
\le \sum_{\set*{i\colon \lambda_{i,\epsilon} \ge \Lambda_1}} g_{T_0}(\lambda_{i,\epsilon})
\\
  \label{e:eigenvalue-weighted-sum-ub}
  &\le \int_{\frac{\Lambda_1}{2}}^\infty g_{T_0}(\lambda)\, dN_\epsilon(\lambda)
  \,.
\end{align}

Observe that $f_{T_0}$ is decreasing on $[\Lambda_1/2, \infty)$ 
  (since $T_0 \ge \frac{2}{\Lambda_1}(1 + \floor{\frac{d}{2}})$), and therefore $g_{T_0}$ is 
decreasing on $[0, \infty)$. Consequently, for each $t \in (0, g_{T_0}(0))$, 
there exists $\lambda_t > 0$ such that 
$\set*{g_{T_0}(\lambda) > t} = [0, \lambda_t)$.
Combining this with Lemma~\ref{l:weylfunction-ub} gives
\begin{align}
  \MoveEqLeft
    \int_0^\infty g_{T_0}(\lambda)\, dN_\epsilon(\lambda)
    = \int_0^{g_{T_0}(0)} N_\epsilon\paren*{\set*{g_{T_0}(\lambda) > t}}\, dt
    = \int_0^{g_{T_0}(0)} N_\epsilon\paren*{[0, \lambda_t)}\, dt\\
    &\overset{\mathclap{\eqref{e:NLepsilon-ub}}}{\le} \;
    C_W \epsilon^{-\frac{d}{2}} 
    \int_0^{g_{T_0}(0)} 
    \paren*{
        \frac{d}{2} \int_0^\infty 
        \one_{\set*{\lambda < \lambda_t}} (1+\lambda)^{\frac{d}{2}-1} d\lambda + 1
    } dt\\
    &\le C_W \epsilon^{-\frac{d}{2}} 
    \paren*{
        \frac{d}{2} \int_0^\infty g_{T_0}(\lambda) (1+\lambda)^{\frac{d}{2}-1} d\lambda 
        + g_{T_0}(0)
    }\\
    \label{e:gN-ub}
    &\le C_W' \epsilon^{-\frac{d}{2}}\,,
\end{align}
where all constants depending on $T_0(\Lambda_1, d)$, $\Lambda_1$, and $d$ 
have been absorbed into $C_W'$.

Combining~\eqref{e:sum-exp-lambda-psiLinfty}, 
\eqref{e:eigenvalue-weighted-sum-ub}, and~\eqref{e:gN-ub}, defining $\tilde C_{\psi} \defeq C_{\Lambda_1} C_W'$, and using the fact that $T_0\Lambda_1 = d+2$, 
we obtain~\eqref{e:psi-Linfty-weighted-sum}.
\end{proof}

\subsection{Proof of Lemmas~\ref{l:weylfunction-ub} and~\ref{l:eigenfunctions-Linfty-ub}}

To finish the proof of Lemma~\ref{l:psi-Linfty-weighted-sum} it remains to prove Lemmas~\ref{l:weylfunction-ub} and~\ref{l:eigenfunctions-Linfty-ub}, which we do here.
Lemma~\ref{l:weylfunction-ub} can be obtained directly from Weyl's law (see for instance Chapter~17.5 in~\cite{Hormander07}, or~\cite{Ivrii16,Sogge17})
In our context however, the operator~$\mathcal L_\epsilon$ depends on~$\epsilon$, and we need asymptotic dependence of the Weyl function as~$\epsilon \to 0$.
A similar result was used in~\cite{ChristieFengEA25}, and we present a proof here for convenience.

\begin{proof}[Proof of Lemma~\ref{l:weylfunction-ub}]
  Define the operator~$\mathcal U_\epsilon \colon L^2(\T^d, \pi_\epsilon) \to L^2(\T^d)$ by
  \begin{equation*}
	  \mathcal U_\epsilon f = \frac{1}{\sqrt{Z_\epsilon}} e^{-U / 2\epsilon} f\,.
  \end{equation*}
  Clearly
  \begin{equation}
    \ip{f, g}_{L^2(\pi_\epsilon)} = \ip{\mathcal U_\epsilon f, \mathcal U_\epsilon g}\,,
  \end{equation}
  and so $\mathcal U$ is an isometry.
  Define the operator~$\mathcal H_\epsilon$ by $\mathcal H_\epsilon \defeq \mathcal U_\epsilon \mathcal L_\epsilon \mathcal U_\epsilon^{-1}$.
  We compute
  \begin{equation*}
	  -\mathcal H_\epsilon f =-\epsilon \lap f +\paren[\Big]{ \frac{1}{4} |\nabla U|^2-\frac{1}{2}\Delta U} f \,.
  \end{equation*}
  Thus~$\mathcal L_\epsilon$ is unitarily equivalent to the operator~$\mathcal H_\epsilon$, and hence the operators~$\mathcal L_\epsilon$ and~$\mathcal H_\epsilon$ have the same spectrum.

Next, for sufficiently large $\gamma = \gamma(\norm{U}_{C^2}) > 0$, we see that the operator $-\mathcal H_{\epsilon, \gamma} \defeq -\mathcal H_{\epsilon} + \gamma I$ is a self-adjoint operator that satisfies the coercivity bound
\begin{equation}
\epsilon\ip{-\Delta f, f}_{L^2} \leq \ip{-\mathcal H_{\epsilon, \gamma}f, f}_{L^2}\,, \quad \forall f\in L^2\,.
\end{equation}
Thus, by the Courant--Fischer min-max principle, we have
\begin{equation}
\epsilon\lambda_k(-\Delta) \leq \lambda_k(-\mathcal H_\epsilon) + \gamma\,,
\end{equation}
and this implies 
\begin{equation}
N_{\mathcal H_\epsilon}(\lambda) \leq N_{\Delta}\paren[\Big]{\frac{\lambda+\gamma}{\epsilon}}
    \leq C(d)\paren[\Big]{\frac{\lambda+\gamma}{\epsilon}}^\frac{d}{2}\,,
\end{equation}
which completes the proof.
\end{proof}

We now turn to the proof of Lemma~\ref{l:eigenfunctions-Linfty-ub}.
This is a standard regularity estimate for~$\mathcal L_\epsilon$ (see for instance~\cite{Hormander07}).
In our situation however, the operator~$\mathcal L_\epsilon$, and the weight~$\pi_\epsilon$ both depend on~$\epsilon$, and our aim is to obtain the bound~\eqref{e:fLinfty-ub-final} with the right~$\epsilon$-dependence.
As a result, we present the proof here, keeping track of the~$\epsilon$-dependence.

\begin{proof}[Proof of Lemma~\ref{l:eigenfunctions-Linfty-ub}]
  For notational simplicity, let~$\psi$ be an eigenfunction of~$-\mathcal L_\epsilon$ with eigenvalue~$\lambda > \Lambda_1$, normalized so that~$\norm{\psi}_{L^2(\pi_\epsilon)} = 1$.
  We claim that for every~$m \in \N$, there exists a constant~$C_m = C_m(\norm{U}_{C^m}, d, \Lambda_1)$ such that
  \begin{equation}\label{e:fHmpi-ub}
    \norm{\psi}_{\dot H^m(\pi_\epsilon)}^2 \leq C_m\paren[\Big]{\frac{\lambda}{\varepsilon}}^m\,.
  \end{equation}
  Lemma~\ref{l:eigenfunctions-Linfty-ub} follows immediately from~\eqref{e:fHmpi-ub}.
  Indeed, since~$\ip{\psi, 1}_{L^2(\pi_\epsilon)} = 0$, the Sobolev embedding theorem and Poincar\'e inequality imply
  \begin{equation}\label{e:fLinfty-ub}
    \norm{\psi}_{L^\infty(\T^d)}^2 \leq C(d)\norm{\psi}_{\dot H^{m}}^2\,,
    \quad \text{for} \quad
    m = 1 + \floor[\Big]{\frac{d}{2}}
    \,.
  \end{equation}
  Also, since
  \begin{align}
  \norm{\psi}_{\dot H^m}^2 &= \sum_{\abs{\alpha}=m}\int_{\T^d} \abs{D^\alpha \psi}^2 dx \\
  \label{e:fHd-ub}
  &=  \sum_{\abs{\alpha}=m}\int_{\T^d} \abs{D^\alpha \psi}^2 \pi_\epsilon \pi_\epsilon^{-1} dx \leq \norm{\psi}_{\dot H^{m}(\pi_\epsilon)}^2 \norm{\pi_\epsilon^{-1}}_\infty\,.
  \end{align}
  Combining~\eqref{e:fHmpi-ub}, \eqref{e:fLinfty-ub}, and~\eqref{e:fHd-ub} implies~\eqref{e:fLinfty-ub-final} as desired.

  It remains to prove~\eqref{e:fHmpi-ub}.
  We do this by induction on~$m$.
  The base case $m=0$ is trivial since $\norm{\psi}_{L^2(\pi_\epsilon)}=1$ by assumption.
  Suppose now~\eqref{e:fHmpi-ub} holds for some~$k = m$.

  Define $l_k = \set{\alpha\in \N^d - 0 \st \abs{\alpha}_{\ell^1}=k}$. Then, we see that
\begin{equation}\label{e:Hk+1def}
\norm{\psi}_{\dot H^{k+1}(\pi_\epsilon)}^2= \sum_{\alpha \in l_k} \sum_{i=1}^d \norm{D^{\alpha+e_i} f}_{L^2(\pi_\epsilon)}^2\,.
\end{equation} 
We now fix $\alpha\in l_k$ and compute $\sum_{i=1}^d \norm{D^{\alpha+e_i} f}_{L^2(\pi_\epsilon)}^2$.
Integrating by parts gives
  \begin{equation}\label{e:Dalphaeif}
\norm{D^{\alpha+e_i} \psi}_{L^2(\pi_\epsilon)}^2
    = -\int_{\T^d} D^\alpha \psi \partial_i^2 D^\alpha \psi \pi_\epsilon  dx - I_1\,,
\end{equation}
where
\begin{equation}
  I_1 = \int_{\T^d} D^\alpha \psi \partial_i D^\alpha \psi \partial_i \pi_\epsilon dx
    = \frac{1}{2} \int_{\T^d} \partial_i (D^\alpha \psi)^2 \partial_i \pi_\epsilon dx
    = - \frac{1}{2} \int_{\T^d} (D^\alpha \psi)^2 \partial_i^2 \pi_\epsilon dx\,.
\end{equation}
Combining this with~\eqref{e:Dalphaeif} and summing over $i$ in~\eqref{e:Hk+1def} produces
\begin{equation}\label{e:sumI2I3}
\sum_{i=1}^d \norm{D^{\alpha+e_i} \psi}_{L^2(\pi_\epsilon)}^2 = I_2 + I_3\,,
\end{equation}
where
\begin{equation}\label{e:I2I3def}
I_2 = -\int_{\T^d} D^\alpha \psi D^\alpha (\Delta \psi) \pi_\epsilon dx\quad\text{and}\quad I_3 = \frac{1}{2} \int_{\T^d} (D^\alpha \psi)^2 \Delta \pi_\epsilon dx\,.
\end{equation}

We first compute $I_2$.
Since~$\psi$ is an eigenfunction of~$\mathcal L_\epsilon$ we note
\begin{equation}\label{e:I2I4}
I_2 = \frac{\lambda}{\epsilon}\int_{\T^d} (D^\alpha \psi)^2 \pi_\epsilon dx -\frac{1}{\epsilon} I_4\,,
\end{equation}
where
\begin{equation}
I_4 = \int_{\T^d} D^\alpha (\nabla U \cdot \nabla \psi) D^\alpha \psi \pi_\epsilon dx\,.
\end{equation}
By considering the case when the differential operator $D^\alpha$ is completely exhausted by hitting $\nabla f$ or not, we can split the integral into two parts. Namely, 
\begin{align}\label{e:I4def}
I_4=I_5+I_6 
\end{align}
where
\begin{align}
I_5 = \int_{\T^d} \nabla U \cdot \nabla (D^\alpha \psi) D^\alpha \psi \pi_\epsilon \quad\text{and}\quad 
I_6 = I_4 - I_5\,.
\end{align}

Integrating~$I_5$ by parts gives
\begin{align}
I_5 = -\int_{\T^d} \Delta U (D^\alpha \psi)^2 \pi_\epsilon dx - I_5 - \int_{\T^d} \nabla U \cdot \nabla \pi_\epsilon (D^\alpha \psi)^2 dx\,, 
\end{align}
and consequently,
\begin{equation}\label{e:I5compute}
I_5 = -\frac{1}{2} \int_{\T^d} (\Delta U \pi_\epsilon + \nabla U \cdot \nabla \pi_\epsilon)(D^\alpha \psi)^2 dx\,.
\end{equation}
Combining~\eqref{e:sumI2I3}, \eqref{e:I2I3def}, \eqref{e:I2I4}, \eqref{e:I4def}, and~\eqref{e:I5compute}, we get
\begin{equation}\label{e:sumDalphaf}
\sum_{i=1}^d \norm{D^{\alpha+e_i} \psi}_{L^2(\pi_\epsilon)}^2 = \frac{\lambda}{\epsilon} \norm{D^\alpha \psi}_{L^2(\pi_\epsilon)}^2 -\frac{1}{\epsilon}I_6+ \frac{1}{2\epsilon}I_7\,,
\end{equation}
where
\begin{equation}
I_7 = \int_{\T^d} (\Delta U \pi_\epsilon + \nabla U \cdot \nabla \pi_\epsilon + \epsilon \Delta \pi_\epsilon) (D^\alpha \psi)^2 dx\,.
\end{equation}
We know~$\pi_\epsilon$ is a stationary solution to the Kolmogorov forward equation
\begin{equation}
  \mathcal L_\epsilon^* \pi_\epsilon
    = \Delta U \pi_\epsilon + \nabla U \cdot \nabla \pi_\epsilon + \epsilon \Delta \pi_\epsilon
    = 0
    \,.
\end{equation}
Hence $I_7=0$.

Moreover,
\begin{equation}
I_6 = \sum_{\beta < \alpha}
  \binom{\alpha}{\beta}
 \int_{\T^d} (D^{\alpha-\beta} \nabla U)  \cdot (D^{\beta} \nabla \psi) D^\alpha \psi \pi_\epsilon dx\,,
\end{equation}
and by using Cauchy--Schwartz, induction hypothesis, and the fact that $\abs{\beta} \leq k-1$,
\begin{align}
\abs{I_6} &\leq \sum_{\beta < \alpha}
\begin{pmatrix}
\alpha\\
\beta
\end{pmatrix} 
\norm{U}_{C^{1+\abs{\alpha-\beta}}} \norm{\psi}_{\dot H^{\abs{\beta}+1}(\pi_\epsilon)} \norm{\psi}_{\dot H^k(\pi_\epsilon)} \\
\label{e:I6ub}
&\leq C(k, \norm{U}_{C^{k+1}}, d, \Lambda_1)\norm{\psi}_{\dot H^k(\pi_\epsilon)}^2\,.
\end{align}
Combining~\eqref{e:Hk+1def}, \eqref{e:sumDalphaf}, and~\eqref{e:I6ub} yields
\begin{align}
\norm{\psi}_{\dot H_{k+1}^2(\pi_\epsilon)}^2 &\leq C(k, \norm{U}_{C^{k+1}}, d, \Lambda_1)\paren[\Big]{\paren[\Big]{\frac{\lambda}{\epsilon}}^{k+1} + \frac{1}{\epsilon}\paren[\Big]{\frac{\lambda}{\epsilon}}^{k}}\\
&\leq C(k, \norm{U}_{C^{k+1}}, d, \Lambda_1) \paren[\Big]{\frac{\lambda}{\epsilon}}^{k+1}\,,
\end{align}
where we used the lower bound~$\lambda \geq \Lambda_1$ and increased~$C(k, \norm{U}_{C^{k+1}}, d, \Lambda_1)$ appropriately to obtain the last inequality.
This concludes the proof.
\end{proof}

\section{Proof of convergence for Autonormalized AIS}\label{s:anais}

We will now prove Theorem~\ref{t:langevinGen-anais}.
As before we assume without loss of generality that~$\epsilon_1 = 1$.
Let $U$ be a double well potential that satisfies Assumptions~\ref{a:criticalpts}--\ref{a:massRatioBound} for some~$\epsilon_{\min} < 1 \leq \epsilon_{\max}$.
Let~$\nu > 0$ be a fixed constant, $K$, $T_0$, and $\bar C_w$ be as Lemma~\ref{l:prodBoundLangevin}.
To prove Theorem~\ref{t:langevinGen-anais}, we will need to bound the product of both~$\norm{P^T_k r_k^2}_\infty$ and $\norm{P^T_k r_k^{-2}}_\infty$.
This is a little stronger than Assumption~\ref{a:r2} which is all that was needed for the proof of Theorem~\ref{t:langevinGen}.
This is our first lemma.

\begin{lemma}\label{l:prodBoundLangevin-anais}
There exists a constant~$C_1(\nu, \epsilon_1)$ such that for every~$T \geq T_0$ we have
  \begin{equation}\label{e:anais-prod-ub}
    \prod_{k = 1}^K
      \paren*{1 \vee \norm{P_{k}^T r_{k}^2}_{L^\infty}
      \norm{P_{k}^T r_{k}^{-2}}_{L^\infty}}
	\leq C_1 \bar C_w^2
    \,.
  \end{equation}
\end{lemma}

\begin{proof}
  Similar to~\eqref{e:rk2pik}, \eqref{e:rk2pik2}, and~\eqref{e:rk2pik3} we compute
  \begin{align}
      \norm{\tilde r_k^{-2}}_{L^1(\pi_k)}
	&= \ip{\tilde r_k^{-2}, \pi_k} = \frac{Z_{k-2}}{Z_k} \,,
	\\
      \ip{\tilde r_k^{-2}, \psi_{2,k}}_{L^2(\pi_k)}
	&= \frac{Z_{k-2}}{Z_k}\ip{\psi_{2,k}, \pi_{k-2}} \,,
	\\
      \abs*{\ip{\tilde r_k^{-2}, \psi_{i,k}}_{L^2(\pi_k)}}
      &\le
      \norm{\tilde r_k^{-2}}_{L^1(\pi_k)}
      \norm{\psi_{i,k}}_{L^\infty} \,.
  \end{align}

  Next we claim that the similar to~\eqref{e:PsiEpPi}, we also have a bound on~$\int_{\mathbb T^d}\psi_{2,\epsilon'}\pi_{\epsilon}\, d x$.
  Explicitly, we claim
\begin{equation}\label{e:PsiEpPrimePi}
      \abs[\big]{
	\ip{\psi_{2,\epsilon'}, \pi_{\epsilon}}
	}
    \leq 
  C_\gamma\paren[\Big]{\exp\paren[\Big]{-\frac{\gamma}{\epsilon}}+\abs[\big]{\pi_{\epsilon'}(\Omega_1)-\pi_{\epsilon}(\Omega_1)}}
\end{equation}
  for every~$0 < \epsilon' < \epsilon \leq 1$, and some (possibly larger) constant~$C_\gamma$ that is independent of~$\epsilon, \epsilon'$.
  The proof of this is almost identical to the proof of~\eqref{e:PsiEpPi} presented in Section 9 of~\cite{HanIyerEA26} and we do not reproduce it here.

  Following the proof of Lemma~\ref{l:eTLk-tilde-rk^2}, and using~\eqref{e:PsiEpPrimePi} in place of~\eqref{e:PsiEpPi} for the second estimate in that proof, we obtain a bound for~$\tilde r_k^{-2}$ similar to~\eqref{e:eTLk-tilde-rk^2-Linfty-ub}.
  Explicitly, for the constant $C_{\psi_2}$ in Property~\ref{property:eigenfunctions}, and for any $T$ satisfying~\eqref{e:Tchoice}, we have
\begin{equation}
    \norm*{P_{k,T}\tilde r_k^{-2}}_{L^\infty}
    \le
    \frac{Z_{k-2}}{Z_k}
    \paren*{
        1
        +
        C_\gamma C_{\psi_2}\paren*{
            \exp\paren*{-\frac{\hat \gamma}{2 \epsilon_k}}
            +
            \abs*{\pi_{k+1}(\Omega_1)-\pi_k(\Omega_1)}
        }
        +
        \frac{1}{K}
    } .
\end{equation}

Combining this estimate with~\eqref{e:eTLk-tilde-rk^2-Linfty-ub}, we obtain
\begin{equation}
\norm*{P_{k,T} r_k^2}_{L^\infty}
\norm*{P_{k,T} r_k^{-2}}_{L^\infty}
\le
\frac{Z_{k-2}Z_{k+2}}{Z_k^2}\Theta(k,k+1)^2,
\end{equation}
where $\Theta$ is defined in~\eqref{e:Theta-def}.

  Since~$\set{1/\epsilon_k}$ are linearly spaced,
  \begin{equation}
    \int_{\T^d} e^{-U / \epsilon_k} \, dx
      = \int_{\T^d} e^{-U / (2\epsilon_{k-2})} e^{-U / (2\epsilon_{k+2})}  \, dx
  \end{equation}
  and so the Cauchy--Schwartz inequality implies
  \begin{equation}
    Z_k^2 \leq Z_{k-2} Z_{k+2}
    \,.
  \end{equation}
Hence
\begin{equation}
\min\set*{
    \frac{Z_{k-2}Z_{k+2}}{Z_k^2},
    \Theta(k,k+1)
}
\ge 1,
\end{equation}
and so
\begin{align}
\prod_{k=1}^K
\paren*{
    1 \vee
    \norm{P_k^T r_k^2}_{L^\infty}
    \norm{P_k^T r_k^{-2}}_{L^\infty}
}
&\le
\frac{Z_{-1}Z_0}{Z_1Z_2}
\frac{Z_{K+1}Z_{K+2}}{Z_{K-1}Z_K}
\paren*{\prod_{k=1}^K \Theta(k,k+1)}^2
\\
&\overset{\eqref{e:prod-Thetak-ub}}{\le}
C_1(\nu) C_\Theta^2 .
\end{align}

Finally, choosing $T_0$ as in~\eqref{e:CwT0} (with the same choice of~$C_T$ in~\eqref{e:CTdef}), choosing $\bar C_w$ as in~\eqref{e:T0-tildeCw-choice}, and using the previous estimate show that~\eqref{e:anais-prod-ub} holds for any~$T \geq T_0$.
This concludes the proof.
\end{proof}

Next we prove Proposition~\ref{p:essANLangevin}, showing that the effective sample size is at least~$N / (C_1 \bar C_w^2)$.

\begin{proof}[Proof of Proposition~\ref{p:essANLangevin}]
For notational convenience, in this proof we use~$w_{k,i}$ to denote the weights~$w_k^i$ returned by Algorithm~\ref{a:aisAutoNormalized}.
  We first note that by the definition of~$w_{k,i}$ and~$\tilde W_k$ in Algorithm~\ref{a:aisAutoNormalized}, we can use the ratio of the \emph{normalized} densities~$r_k$, instead of that of the \emph{unnormalized} densities~$\tilde r_k$, because the same normalizing constant~$Z_{k+1}/Z_k$ appears in both the numerator and the denominator.
  Explicitly, we note
\begin{equation}
w_{k, i}
  = \frac{w_{k-1,i} \tilde r_k(X_{k}^i)}{\tilde W_k}
  = \frac{w_{k-1,i} r_k(X_{k}^i)}{W_k}\,,
  \quad\text{where}\quad W_k = \sum_{i=1}^N w_{k-1,i} r_k(X_k^i)\,.
\end{equation}
  By Jensen's inequality applied to the convex function $x \mapsto 1/ x^{2}$, we obtain
\begin{equation}
  \frac{1}{W_k^{2}}
  = \frac{1}{\bigl( \sum_{i=1}^N w_{k-1,i} r_k(X_k^i) \bigr)^{2}}
\leq
  \sum_{i=1}^N \frac{w_{k-1,i}}{ r_k^{2}(X_k^i)}
  \,.
\end{equation}
Hence,
\begin{align}
\sum_{i=1}^N w_{k,i}^2
  &= \frac{1}{W_k^{2}} \sum_{i=1}^N w_{k-1,i}^2 r_k^2(X_k^i) \\
&\leq
  \Biggl( \sum_{j=1}^N \frac{w_{k-1,j}}{r_k^{2}(X_k^j)} \Biggr)
\Biggl( \sum_{i=1}^N w_{k-1,i}^2 r_k^2(X_k^i) \Biggr)
  \\
  \label{e:ssw-cubic-estimate}
&=
\sum_{i=1}^N w_{k-1,i}^3
+
\sum_{i=1}^N \sum_{j \neq i}
w_{k-1,i}^2 w_{k-1,j}
\, r_k^2(X_k^i) r_k^{-2}(X_k^j)\,.
\end{align}

  Let~$\E_{k-1}$ denote the conditional expectation with respect to the~$\sigma$-algebra generated by~$\set{w_{k-1, i}, X_{k-1}}_{1 \leq i \leq N}$.
  Applying~$\E_{k-1}$ to both sides of~\eqref{e:ssw-cubic-estimate}, and using the conditional independence of $X_{k}^i$ and $X_{k}^j$ for $j \neq i$, we obtain
\begin{align}
  \MoveEqLeft
  \E_{k-1} \Biggl[ \sum_{i=1}^N w_{k, i}^2 \Biggr]
\\
  & \leq\sum_{i=1}^N w_{k-1,i}^3
+
\sum_{i=1}^N \sum_{j \neq i}
w_{k-1,i}^2 w_{k-1,j}
  \, \E_{k-1}\bigl[r_k^2(X_{k}^i)\bigr]
  \, \E_{k-1}\bigl[r_k^{-2}(X_{k}^j)\bigr] \\
&=
\sum_{i=1}^N w_{k-1,i}^3
+
\sum_{i=1}^N \sum_{j \neq i}
w_{k-1,i}^2 w_{k-1,j}
\,
\bigl(P_k^T r_k^2\bigr)(X_{k-1}^i)
\,
\bigl(P_k^T r_k^{-2}\bigr)(X_{k-1}^j) \\
&\leq
\paren*{
1 \vee
\norm{P_{k}^T r_{k}^2}_{L^\infty}
\norm{P_{k}^T r_{k}^{-2}}_{L^\infty}
}
\sum_{i=1}^N w_{k-1,i}^2.
\end{align}
Taking expectations on both sides and iterating gives
\begin{align}
  \E \Biggl[ \sum_{i=1}^N w_{K, i}^2 \Biggr]
    &\leq
      \paren*{
      \prod_{k = 1}^K
      1 \vee
      \norm{P_{k}^T r_{k}^2}_{L^\infty}
      \norm{P_{k}^T r_{k}^{-2}}_{L^\infty}
      }
  \sum_{i=1}^N w_{0,i}^2.
  \\
  \label{e:ssw1}
  &=
    \frac{1}{N}
      \paren*{
      \prod_{k = 1}^K
      1 \vee
      \norm{P_{k}^T r_{k}^2}_{L^\infty}
      \norm{P_{k}^T r_{k}^{-2}}_{L^\infty}
      }
  \,,
\end{align}
where the last equality holds because~$w_{0, i} = 1/N$.

By increasing~$\hat C_T$, if necessary, the choice of~$T$ in~\eqref{e:T-choice-emp-anais} ensures that~$T \geq T_0$.
Therefore, Lemma~\ref{l:prodBoundLangevin-anais} applies, and~\eqref{e:anais-prod-ub} holds, and~\eqref{e:ssw1} implies~\eqref{e:sum-w-k1-to-k} as desired.
The bound~\eqref{e:essAN} then follows from~\eqref{e:sum-w-k1-to-k} and Jensen's inequality.
\end{proof}

It remains to prove Theorem~\ref{t:langevinGen-anais}.
To this end, we introduce the following notation.
For every bounded test function~$h$, and every~$1\leq k\leq K$, define
\begin{align}
\Err_{k, T}(h) &= \norm*{\ip*{h, \mu_{k,T}- \pi_{k}}}_{L^2(\P)}\,,\\
\Err_{k+1, 0}(h) &= \norm*{\ip*{h, \mu_{k+1, 0} - \pi_{k+1}}  }_{L^2(\P)}
\end{align}
where
\begin{equation}
\mu_{k, T}\defeq \sum_{i=1}^N w_{k-1}^i\delta_{X_k^i}\quad\text{and}\quad \mu_{k+1,0} \defeq \sum_{i=1}^N w_k^i  \delta_{X_k^i}\,.
\end{equation}
We first state and prove the following lemma that connects the error before and after the reweigting.

\begin{lemma}\label{l:k+10-to-kT}
For any~$1\leq k \leq K$, 
\begin{equation}\label{e:k+10-to-kT}
    \Err_{k+1,0}(h) \leq \norm{h}_{L^{\infty}} \Err_{k,T}(r_k) + \Err_{k,T}(r_k h)
\end{equation}
\end{lemma}

\begin{proof}
Fix $x_1, x_2, \ldots, x_N \in \T^d$ and define the normalization constant $R_k$ and the updated weight $w_{k}^i$ by
\begin{equation}\label{e:Rk-def}
	R_k \defeq \sum_{i=1}^N w_{k-1}^i r_k(x_i)\,,\quad\text{and}\quad w_{k}^i = \frac{w_{k-1}^i r_k(x_i)}{R_{k}}\,.
\end{equation}
    Adding and subtracting the same term and using the triangle inequality, we obtain
    \begin{multline}\label{e:errk+10h-to-errkTr}
        \abs*{\sum_{i=1}^N w_{k}^i h(x_i) - \ip{h, \pi_{k+1}}}\\
        \leq \abs*{\sum_{i=1}^N w_{k}^i h(x_i) - R_k\sum_{i=1}^N w_{k}^i  h(x_i)} + \abs*{\sum_{i=1}^N w_{k-1}^i r_k(x_i) h(x_i) - \ip{h, \pi_{k+1}}}\,.
    \end{multline}

The first term can be estimated by
\begin{equation}\label{e:errk+10h-first-term}
\abs*{\sum_{i=1}^N w_{k}^i h(x_i) - R_k\sum_{i=1}^N w_{k}^i  h(x_i)} \leq \abs*{1-R_k}\norm{h}_\infty\,,
\end{equation}
and hence combining~\eqref{e:errk+10h-to-errkTr}, \eqref{e:errk+10h-first-term}, the fact that $\ip{r_k, \pi_k} = 1$ and $\ip{r_k h, \pi_k} = \ip{h, \pi_{k+1}}$, and substituting~$X_k^i$ for~$x_i$ yield~\eqref{e:k+10-to-kT}.
\end{proof}

Next, we state four lemmas required for the proof of Theorem~\ref{t:langevinGen-anais}. Their proofs are identical or closely follow those in~\cite{HanIyerEA26}, and are therefore omitted. For convenience, we indicate in each statement the corresponding result in~\cite{HanIyerEA26}.

\begin{lemma}[Lemma 4.10, \cite{HanIyerEA26}]
\label{l:errkth-iter}
Assume that for each~$k\in\set{2, \ldots, K}$, the law of $X_{k}^1$ has density $q_{k}$. Then for any bounded test function $h$ and~$2\leq k\leq K$,
\begin{multline}\label{e:errkth-iter}
	\Err_{k, T}(h) \leq e^{-\lambda_{2, k}T}
	  \abs[\big]{ \ip{ h\psi_{2,k}, \pi_{k}} }\Err_{k, 0}(\psi_{2, k})
      \\
	  + \E\brak[\Big]{\sum_{i=1}^N w_{k-1, i}^2}^\frac{1}{2}\norm{h}_\infty
	+ \sqrt{N}  \norm[\Big]{\frac{q_{k-1}}{\pi_{k}}}_\infty^\frac{1}{2} e^{-\Lambda T}\norm{h}_\infty  \,.
\end{multline}
\end{lemma}

\begin{lemma}[Lemma 4.12, \cite{HanIyerEA26}]\label{l:errk10psi-iter}
For any $\alpha > 0$, there exist constants $C_\alpha = C_\alpha(\alpha, U)>0$ (depending on $\alpha$) and $\hat C_N = \hat C_N(U, \nu) > 1$ (independent of $\alpha$) such that for any $\delta>0$, if 
\begin{equation}\label{e:NT-choice1}
N \geq \hat C_N \frac{K^2}{\delta^2}\,,\quad
T \geq C_\alpha \paren*{K^{(1+\alpha)\hat \gamma_r} + \frac{1}{\epsilon} + \log\paren*{\frac{1}{\delta}} + \log\paren*{N}}\,,
\end{equation}
then for each $2\leq k\leq K-1$, we have
\begin{equation}\label{e:Err-recursion}
\Err_{k+1, 0}(\psi_{2, k+1}) \leq \beta_k \Err_{k,0}(\psi_{2,k}) + c_k\,.
\end{equation}

Here, the constants $\beta_k, c_k$ are such that for every $k\in \set{2,\ldots, K-1}$, we have 
\begin{align}\label{e:bprod-bound}
\prod_{j=k}^{K-1} \beta_j &\leq C_\beta\,,\\
\label{e:c-bound}
c_k&\leq \frac{\delta}{K}\,,
\end{align}
for some dimensional constant $C_\beta > 1$ (independent of $\alpha$, $\delta$).
\end{lemma}

\begin{lemma}[Lemma 4.11, \cite{HanIyerEA26}]
\label{l:q/pi-rinv-ub}
For every $2\leq k\leq K$, let $q_{k}$ be the probability density function of $X_{k}^1$. For any $\hat T_0>0$, there exists a constant $C_q=C_q(U, \hat T_0)$ such that if $T \geq \hat T_0$, then
\begin{equation}\label{e:q-over-pi-linfty-ub}
\norm[\Big]{\frac{q_{k-1}}{\pi_k}}_\infty \leq C_q\exp\paren*{\paren*{\frac{1}{\epsilon}-1}\norm*{U}_\infty}\,.
\end{equation}
\end{lemma}

\begin{lemma}[Lemma 4.13, \cite{HanIyerEA26}]\label{l:first-level}
There exists a constant $\hat C_1=\hat C_1(U)$ such that for any $\delta>0$ and
\begin{equation}\label{e:NT-choice-first-level}
N\geq \frac{\hat C_N}{\delta^2} \,,\quad T\geq \hat C_1\paren*{\log\paren*{\frac{1}{\delta}} + 1 + \log\paren*{N}}\,,
\end{equation}
we have
\begin{equation}\label{e:first-level-small-error}
\Err_{2,0}(\psi_{2,2})\leq \delta\,.
\end{equation}
Here, $\hat C_N$ is the same constant defined in Lemma~\ref{l:errk10psi-iter}.
\end{lemma}

\begin{proof}[Proof of Theorem~\ref{t:langevinGen-anais}]
Let $C_\beta$ be as in Lemma~\ref{l:errk10psi-iter}, and define
\begin{equation}
    C_r = 1 \vee \paren*{\max_{1\leq k\leq K} \norm{r_k}_\infty}\,.
\end{equation}
Fix $\delta > 0$ and set $\tilde{\delta}\defeq \frac{\delta}{8C_\beta C_r}$. Fix $\alpha > 0$ and suppose that
\begin{align}
\label{e:Nlarge-final}
N &\geq \frac{\max\set*{\hat C_N,\, C_1\bar C_w^{2} C_\beta^{-2} }}{\tilde \delta^2}K^2\,,\\
\label{e:Tlarge-final}
T &\geq \max \Biggl\{
  \begin{aligned}[t]
    &C_\alpha \Bigl( K^{(1+\alpha)\hat{\gamma}_r} + \frac{1}{\epsilon} + \log \frac{1}{\tilde{\delta}} + \log N \Bigr)
    \,,
    \\
      &\hat C_1 \Bigl( \log \frac{1}{\tilde{\delta}} + 1 + \log N \Bigr)\,, 
       C_T\paren*{\frac{1}{\epsilon} + \log K}\,,
    \\
    &\frac{1}{2\Lambda}\frac{\|U\|_\infty}{\epsilon}
      + \frac{1}{\Lambda} \log \biggl( \frac{C_q^{1/2}}{\tilde{\delta}} \sqrt{N} \biggr)\,,
      ~1
    \Biggr\},
  \end{aligned}
\end{align}
where~$\bar C_w, C_T$ are as in Lemma~\ref{l:prodBoundLangevin}, $C_1$ is as in Lemma~\ref{l:prodBoundLangevin-anais}, $C_\alpha$ and $\hat C_N$ are as in Lemma~\ref{l:errk10psi-iter}, $C_q = C_q(U,1)$ is as in Lemma~\ref{l:q/pi-rinv-ub}, and~$\hat C_1$ is as in Lemma~\ref{l:first-level}. 
In particular, choosing $\hat C_T$ and $C_N$ in~\eqref{e:T-choice-emp-anais} and~\eqref{e:N-choice-emp-anais} sufficiently large ensures that the lower bounds in~\eqref{e:Nlarge-final} and~\eqref{e:Tlarge-final} are satisfied. 
Moreover, with this choice of~$N$ and~$T$, $N$ is sufficiently large to satisfy~\eqref{e:NT-choice1} and~\eqref{e:NT-choice-first-level} with~$\delta=\tilde \delta$ and $T$ is sufficiently large to satisfy $T \geq \max\set{T_0,1}$, \eqref{e:NT-choice1}, and~\eqref{e:NT-choice-first-level} with~$\delta=\tilde \delta$. 
Hence, Proposition~\ref{p:essANLangevin} holds, Lemmas~\ref{l:errk10psi-iter} and~\ref{l:first-level} hold with~$\delta=\tilde \delta$, and Lemma~\ref{l:q/pi-rinv-ub} holds with~$\hat T_0 = 1$.

Using Lemmas~\ref{l:k+10-to-kT}, \ref{e:errkth-iter}, and Proposition~\ref{p:essANLangevin}, we obtain that for any bounded test function~$h$,
\begin{align}
  \MoveEqLeft
    \Err_{K+1,0}(h)
  \\
  \label{e:errMTh}
    &\leq \Biggl( \Err_{K,0}(\psi_{2,K}) + \frac{C_1^\frac{1}{2}\bar C_w}{\sqrt{N}} + \sqrt{N}
	\Bigl\| \frac{q_{K-1}}{\pi_K} \Bigr\|_\infty^{1/2} e^{-\Lambda T} \Biggr)
\norm{r_K}_\infty \norm{h}_\infty \,.
\end{align}
We estimate the terms on the right hand side separately.
Applying~\eqref{e:Err-recursion} for every $2 \leq k \leq K-1$, together with~\eqref{e:first-level-small-error}, yields
\begin{align}\label{e:errM0psi2M}
\Err_{K,0}(\psi_{2,K})
&\leq \Biggl( \prod_{j=2}^{K-1} \beta_j \Biggr) \Err_{2,0}(\psi_{2,2})
+ \sum_{k=2}^{K-2} c_k \Biggl( \prod_{j=k+1}^{K-1} \beta_j \Biggr)
+ c_{K-1} \\
&\leq C_\beta \tilde{\delta}
+ \sum_{k=2}^{K-2} C_\beta \frac{\tilde{\delta}}{K}
+ \frac{\tilde{\delta}}{K}
\leq 2C_\beta \tilde{\delta}
= \frac{\delta}{4C_r}\,.
\end{align}

Moreover,
\begin{equation}\label{e:Mxi-over-sqrt-N}
\frac{C_1^\frac{1}{2} \bar C_w }{\sqrt{N}} 
\overset{\eqref{e:Nlarge-final}}{\leq} C_\beta \tilde \delta 
= \frac{\delta}{8C_r}\,.
\end{equation}

Finally, using Lemma~\ref{l:q/pi-rinv-ub}, we obtain
\begin{equation}\label{e:eigen-larger-than-3}
\sqrt{N} \Bigl\| \frac{q_{K-1}}{\pi_{K}} \Bigr\|_\infty^{1/2} e^{-\Lambda T}
\leq \sqrt{N} C_q^{1/2}
\exp \Bigl( \frac{\|U\|_\infty}{2\epsilon} - \Lambda T \Bigr)
\overset{\eqref{e:Tlarge-final}}{\leq}
\tilde{\delta}
\leq \frac{\delta}{8C_r}\,.
\end{equation}
Combining~\eqref{e:errMTh}, \eqref{e:errM0psi2M}, \eqref{e:Mxi-over-sqrt-N}, and~\eqref{e:eigen-larger-than-3}, we obtain
\begin{equation}
    \Err_{K+1,0}(h) \leq \frac{\delta}{C_r}\norm{h}_\infty \norm{r_K}_\infty \leq \delta \norm{h}_\infty\,,
\end{equation}
which completes the proof of~\eqref{e:an-ais-error-bound}.
Equation~\eqref{e:empHs-anais} directly follows from~\eqref{e:an-ais-error-bound} as in the proof of~\eqref{e:empHs} in Theorem~\ref{t:ais}.
\end{proof}

\bibliographystyle{halpha-abbrv} 
\bibliography{gautam-refs1,gautam-refs2,preprints}

\end{document}